\documentclass[twoside,leqno,10pt, A4]{amsart}
\usepackage{amsfonts}
\usepackage{amsmath}
\usepackage{amscd}
\usepackage{amssymb}
\usepackage{amsthm}
\usepackage{amsrefs}
\usepackage{latexsym}
\usepackage{mathrsfs}
\usepackage{bbm}
\usepackage{amscd}
\usepackage{amssymb}
\usepackage{amsthm}
\usepackage{amsrefs}
\usepackage{latexsym}
\usepackage{mathrsfs}
\usepackage{bbm}
\usepackage{enumerate}
\usepackage{graphicx}
\allowdisplaybreaks
\usepackage{color}
\begin{document}

\newtheorem{theorem}[subsection]{Theorem}
\newtheorem{proposition}[subsection]{Proposition}
\newtheorem{lemma}[subsection]{Lemma}
\newtheorem{corollary}[subsection]{Corollary}
\newtheorem{conjecture}[subsection]{Conjecture}
\newtheorem{prop}[subsection]{Proposition}
\newtheorem{defin}[subsection]{Definition}

\numberwithin{equation}{section}
\newcommand{\mr}{\ensuremath{\mathbb R}}
\newcommand{\mc}{\ensuremath{\mathbb C}}
\newcommand{\N}{\mathbb{N}}
\newcommand{\dif}{\mathrm{d}}
\newcommand{\intz}{\mathbb{Z}}
\newcommand{\ratq}{\mathbb{Q}}
\newcommand{\natn}{\mathbb{N}}
\newcommand{\comc}{\mathbb{C}}
\newcommand{\rear}{\mathbb{R}}
\newcommand{\prip}{\mathbb{P}}
\newcommand{\uph}{\mathbb{H}}
\newcommand{\fief}{\mathbb{F}}
\newcommand{\majorarc}{\mathfrak{M}}
\newcommand{\minorarc}{\mathfrak{m}}
\newcommand{\sings}{\mathfrak{S}}
\newcommand{\fA}{\ensuremath{\mathfrak A}}
\newcommand{\mn}{\ensuremath{\mathbb N}}
\newcommand{\mq}{\ensuremath{\mathbb Q}}
\newcommand{\half}{\tfrac{1}{2}}
\newcommand{\f}{f\times \chi}
\newcommand{\summ}{\mathop{{\sum}^{\star}}}
\newcommand{\chiq}{\chi \bmod q}
\newcommand{\chidb}{\chi \bmod db}
\newcommand{\chid}{\chi \bmod d}
\newcommand{\sym}{\text{sym}^2}
\newcommand{\hhalf}{\tfrac{1}{2}}
\newcommand{\sumstar}{\sideset{}{^*}\sum}
\newcommand{\sumprime}{\sideset{}{'}\sum}
\newcommand{\sumprimeprime}{\sideset{}{''}\sum}
\newcommand{\sumflat}{\sideset{}{^\flat}\sum}
\newcommand{\shortmod}{\ensuremath{\negthickspace \negthickspace \negthickspace \pmod}}
\newcommand{\V}{V\left(\frac{nm}{q^2}\right)}
\newcommand{\sumi}{\mathop{{\sum}^{\dagger}}}
\newcommand{\mz}{\ensuremath{\mathbb Z}}
\newcommand{\leg}[2]{\left(\frac{#1}{#2}\right)}
\newcommand{\muK}{\mu_{\omega}}
\newcommand{\thalf}{\tfrac12}
\newcommand{\lp}{\left(}
\newcommand{\rp}{\right)}
\newcommand{\Lam}{\Lambda_{[i]}}
\newcommand{\lam}{\lambda}
\newcommand{\af}{\mathfrak{a}}
\newcommand{\sw}{S_{[i]}(X,Y;\Phi,\Psi)}
\newcommand{\lz}{\left(}
\newcommand{\pz}{\right)}
\newcommand{\bfrac}[2]{\lz\frac{#1}{#2}\pz}
\newcommand{\odd}{\mathrm{\ primary}}
\newcommand{\even}{\text{ even}}
\newcommand{\res}{\mathrm{Res}}
\newcommand{\sumn}{\sumstar_{(c,1+i)=1}  w\left( \frac {N(c)}X \right)}
\newcommand{\lab}{\left|}
\newcommand{\rab}{\right|}
\newcommand{\Go}{\Gamma_{o}}
\newcommand{\Ge}{\Gamma_{e}}
\newcommand{\M}{\widehat}
\def\su#1{\sum_{\substack{#1}}}
\newcommand{\Echar}{\mathbb{E}^{\text{char}}}
\newcommand{\E}{\mathbb{E}}
\newcommand{\p}{\mathbb{P}}

\theoremstyle{plain}
\newtheorem{conj}{Conjecture}
\newtheorem{remark}[subsection]{Remark}

\newcommand{\pfrac}[2]{\left(\frac{#1}{#2}\right)}
\newcommand{\pmfrac}[2]{\left(\mfrac{#1}{#2}\right)}
\newcommand{\ptfrac}[2]{\left(\tfrac{#1}{#2}\right)}
\newcommand{\pMatrix}[4]{\left(\begin{matrix}#1 & #2 \\ #3 & #4\end{matrix}\right)}
\newcommand{\ppMatrix}[4]{\left(\!\pMatrix{#1}{#2}{#3}{#4}\!\right)}
\renewcommand{\pmatrix}[4]{\left(\begin{smallmatrix}#1 & #2 \\ #3 & #4\end{smallmatrix}\right)}
\def\en{{\mathbf{\,e}}_n}

\newcommand{\ppmod}[1]{\hspace{-0.15cm}\pmod{#1}}
\newcommand{\ccom}[1]{{\color{red}{Chantal: #1}} }
\newcommand{\acom}[1]{{\color{blue}{Alia: #1}} }
\newcommand{\alexcom}[1]{{\color{green}{Alex: #1}} }
\newcommand{\hcom}[1]{{\color{brown}{Hua: #1}} }

\makeatletter
\def\widebreve{\mathpalette\wide@breve}
\def\wide@breve#1#2{\sbox\z@{$#1#2$}%
     \mathop{\vbox{\m@th\ialign{##\crcr
\kern0.08em\brevefill#1{0.8\wd\z@}\crcr\noalign{\nointerlineskip}%
                    $\hss#1#2\hss$\crcr}}}\limits}
\def\brevefill#1#2{$\m@th\sbox\tw@{$#1($}%
  \hss\resizebox{#2}{\wd\tw@}{\rotatebox[origin=c]{90}{\upshape(}}\hss$}
\makeatletter

\title[Upper Bounds for low moments of twisted Fourier coefficients of modular forms]{Upper Bounds for low moments of twisted Fourier coefficients of modular forms}

%%\date{\today}
\author[P. Gao]{Peng Gao}
\address{School of Mathematical Sciences, Beihang University, Beijing 100191, China}
\email{penggao@buaa.edu.cn}

\author[X. Wu]{Xiaosheng Wu}
\address {School of Mathematics, Hefei University of Technology, Hefei 230009,
China.}
\email {xswu@amss.ac.cn}

\begin{abstract}
   For any large prime $q$, $1 \leq x \leq q$ and any real $0 \leq k  \leq 1$, we prove an upper bound for the following $2k$-th moment
    \begin{equation*}
    \sum_{\substack{\chi \shortmod q}} \Big| \sum_{n\leq x} \chi(n)\lambda(n)\Big|^{2k},
\end{equation*}
  where $\lambda(n)$ denotes the Fourier coefficients of a fixed modular form. In particular, our result implies that $$\displaystyle \frac 1{q-1}\sum_{\substack{\chi \shortmod q}} \Big| \sum_{n\leq x} \chi(n)\lambda(n)\Big|= o(\sqrt{x}),$$ when both $x$ and $q/x$ tend to infinity with $q$. 
\end{abstract}

\maketitle

\noindent {\bf Mathematics Subject Classification (2010)}: 11L40, 11M06  \newline

\noindent {\bf Keywords}: Dirichlet characters, modular $L$-functions, Fourier coefficients, upper bounds

\section{Introduction}\label{sec1}

The study on character sums and related topics has received a lot of attention in number theory as they have many important applications. To better understand the behaviors of the character sums, one may often model them using random multiplicative functions. For example, one may simulate the behaviour of any Dirichlet character using a Steinhaus random multiplicative function $f(n)$ defined by $f(n) = \prod_{p^{a} || n} f(p)^{a}$ for all $n \in \N$, where $(f(p))_{p \; \text{prime}}$ is a sequence of independent random variables, each distributed uniformly on the complex unit circle. Here and throughout the paper, we reserve the letter $p$ for a prime number.
  
  In \cite{Harper20}, A. J. Harper determined the order of magnitude of the low moments of sums of Steinhaus random multiplicative functions. Using the ideas in \cite{Harper20}, Harper was further able to show in \cite[Theorem 3]{Harper23} that for large primes $q$, uniformly for any $1 \leq x \leq q$, any $0 \leq k \leq 1$, and any multiplicative function $u(n)$ that has absolute value $1$ on primes and absolute value at most $1$ on prime powers, 
\begin{align}
\label{genlowmoment} 
 \frac{1}{\varphi(q)} \sum_{\chi \shortmod q} \Big| \sum_{n \leq x}  \chi(n)u(n)\Big| ^{2k} \ll \Biggl(\frac{x}{1 + (1-k)\sqrt{\log\log(10L)}} \Biggr)^k, 
\end{align}
where $\varphi(q)$ is the Euler totient function and we set $L=L_q=  \min\{x,q/x\}$ throughout the paper. 

  Setting $u(n)=1$ in \eqref{genlowmoment} recovers \cite[Theorem 1]{Harper23}, which asserts that under the same notation used in \eqref{genlowmoment},  
\begin{align}
\label{genlowmomentDC} 
     \frac{1}{\varphi(q)} \sum_{\chi \shortmod q}  \Big| \sum_{n\leq x} \chi(n) \Big|^{2k} \ll \Biggl(\frac{x}{1 + (1-k)\sqrt{\log\log(10L)}} \Biggr)^k, 
\end{align}
 The above establishes upper bounds for low moments of Dirichlet character sums and it is noted in \cite{Harper23} that these bounds are optimal conjecturally. Note also that the special case $k=1/2$ illustrates that when both $x$ and $q/x$ tend to infinity with $q$, 
\begin{equation}
\label{squareroot}
     \frac{1}{\varphi(q)} \sum_{\chi \shortmod q}  \Big| \sum_{n\leq x} \chi(n) \Big| =o(\sqrt{x}). 
\end{equation}
  The above implies that the sums $\sum_{n\leq x} \chi(n)$  typically exhibit better than square-root cancellation. 

 On the other hand, one may also study high moments of  Dirichlet character sums. In \cite{Szab},  B. Szab\'o  proved that under the generalized Riemann hypothesis (GRH) that for a fixed real number $k>2$ and a large integer $q$, we have for $2 \leq x \leq q$,
\begin{align}
\label{genJacobiupper}
  \frac{1}{\varphi(q)} \sum_{\chi\in X_q^*}\bigg|\sum_{n\leq x} \chi(n)\bigg|^{2k} \ll_k  x^k \left(\min \left(\log x, \log \frac {2q}{x} \right) \right)^{(k-1)^2},
\end{align}
  where $X_q^*$ denotes the set of primitive Dirichlet characters modulo $q$. Here we remark that the proofs of \eqref{genlowmomentDC} and \eqref{genJacobiupper} are based on different approaches. The former builds on many strategies introduced in \cite{Harper20} to study on low moments of random multiplicative functions while the latter relies on results concerning sharp upper bounds on shifted moments of Dirichlet $L$-function on the critical line. We also note that it was shown in \cite[Theorem 1]{Szab24} that the bounds given in \eqref{genJacobiupper} are optimal under GRH for primes $q$. \newline

  Instead of Dirichlet character sums, one may consider other types of sums. For example, the one given on the left-hand side of \eqref{genlowmoment} involving with a general multiplicative function $u(n)$. When $u(n)$ equals the Fourier coefficient $\lambda(n)$ of a fixed holomorphic Hecke eigenform $f$ of weight $\kappa \equiv 0 \pmod 4$ for the full modular group $SL_2 (\mathbb{Z})$, upper and lower bounds for high moments ($k>2$) of the left-hand side expression in \eqref{genlowmoment} were obtained in \cites{G&Zhao24-12, G&Zhao25-11}. Here we recall that the Fourier expansion of $f$ at infinity is given by
\[
f(z) = \sum_{n=1}^{\infty} \lambda (n) n^{\frac{\kappa -1}{2}} e(nz), \quad \mbox{where} \quad e(z) = \exp (2 \pi i z).
\]
  
  It is the aim of this paper is to obtain upper bounds for low moments of the left-hand side expression in \eqref{genlowmoment} with $u(n)=\lambda(n)$ there. Our result is as follows.
\begin{theorem}
\label{lowerboundsfixedmodmean}
With the notation as above. Let $q$ be a large prime number. We have for $1 \leq x \leq q$ and any real $0 \leq k \leq 1$,
\begin{align*}
%%\label{mainestimation}
 \frac{1}{\varphi(q)} \sum_{\chi \shortmod q}  \Big| \sum_{n\leq x} \chi(n)\lambda(n) \Big|^{2k} \ll \Biggl(\frac{x}{1 + (1-k)\sqrt{\log\log(10L)}} \Biggr)^k. 
\end{align*}
\end{theorem}

  We note that similar to \eqref{squareroot}, Theorem \ref{lowerboundsfixedmodmean} implies that for any fixed $0 < k < 1$ and any $x$ such that $x$ and $q/x$ both tend to infinity with $q$, we have 
\begin{align*}
%%\label{genlowmomentDC} 
 \frac{1}{q-1} \sum_{\chi \shortmod q}\Big| \sum_{n \leq x} \chi(n)\lambda(n)\Big| ^{2k} = o(\sqrt{x}),
\end{align*}
  so that the sums $\sum_{n\leq x} \chi(n)\lambda(n)$ also typically exhibit better than square-root cancellation. 
 
  Note that $\lambda(n)$ does not satisfy the condition imposed for $u(n)$ in the establishment of \eqref{genlowmoment}, hence the assertion of Theorem \eqref{lowerboundsfixedmodmean} does not follow \eqref{genlowmoment}. Our proof of Theorem \ref{lowerboundsfixedmodmean} is nevertheless achieved by extending the treatments in the proof of \cite[Theorem 3]{Harper23}.

\section{Preliminaries}
\label{sec 2}

  Before we embark on the proof of Theorem \ref{lowerboundsfixedmodmean}, we first gather some auxiliary results. 

\subsection{Cusp form $L$-functions}
\label{sec:cusp form}

    We reserve the letter $p$ for a prime number throughout in this paper.  Recall that $f$ is a fixed holomorphic Hecke eigenform $f$ of weight $\kappa \equiv 0 \pmod 4$ for the full modular group $SL_2 (\mathbb{Z})$. The associated modular $L$-function $L(s, f)$ for $\Re(s)>1$ is then defined as
\begin{align*}
%%\label{Lphichi}
L(s, f ) &= \sum_{n=1}^{\infty} \frac{\lambda(n)}{n^s}
 = \prod_{p} \left(1 + \sum^{\infty}_{i=1}\frac{\lambda^i(p)}{p^{is}}\right)=\prod_{p} \left(1 - \frac{\alpha_p }{p^s} \right)^{-1}\left(1 - \frac{\beta_p }{p^s} \right)^{-1}.
\end{align*}
  It follows that $\lambda(n)$ is multiplicative that satisfies for any prime $p$ and any non-negative integer $m$,
\begin{align}
\label{Lambdapkrel}
 \lambda^m(p)=\sum^{m}_{j=0}\alpha^{m-j}_p\beta^{j}_p.  
\end{align}

  Moreover, by Deligne's proof \cite{D} of the Weil conjecture, we have
\begin{align}
\label{alpha}
|\alpha_{p}|=|\beta_{p}|=1, \quad \alpha_{p}\beta_{p}=1.
\end{align}
 We then deduce that $\lambda(n) \in \mr$ such that $\lambda (1) =1$ and
\begin{align}
\label{lambdabound}
\begin{split}
  |\lambda(n)| \leq d(n) \ll n^{\varepsilon},
\end{split}
\end{align}
 where $d(n)$ is the number of positive divisors $n$ and the last estimation above follows from \cite[Theorem 2.11]{MVa1}.  We also note  that $|\lambda(n)|^2=\lambda^2(n) \geq 0$ as $\lambda(n)$ is real. 

  The symmetric square $L$-function $L(s, \operatorname{sym}^2 f)$ of $f$ is defined for $\Re(s)>1$ by
 (see \cite[p. 137]{iwakow} and \cite[(25.73)]{iwakow})
\begin{align}
\label{Lsymexp}
\begin{split}
 L(s, \operatorname{sym}^2 f)=& \prod_p(1-\alpha^2_pp^{-s})^{-1}(1-p^{-s})^{-1}(1-\beta^2_pp^{-s})^{-1} \\
 = & \zeta(2s) \sum_{n \geq 1}\frac {\lambda(n^2)}{n^s}=\prod_{p}\left( 1-\frac {\lambda(p^2)}{p^s}+\frac {\lambda(p^2)}{p^{2s}}-\frac {1}{p^{3s}} \right)^{-1}.
\end{split}
\end{align}
 A result of G. Shimura \cite{Shimura} implies that $L(s, \operatorname{sym}^2 f)$ has no pole at $s=1$. Furthermore, the corresponding completed symmetric square $L$-function
\begin{align*}
%%\label{Lambdafdef}
 \Lambda(s, \operatorname{sym}^2 f)=& \pi^{-3s/2}\Gamma \Big( \frac {s+1}{2} \Big)\Gamma \Big(\frac {s+\kappa-1}{2}\Big) \Gamma \Big(\frac {s+\kappa}{2}\Big) L(s, \operatorname{sym}^2 f)
\end{align*}
  is entire and satisfies the functional equation $\Lambda(s, \operatorname{sym}^2 f)=\Lambda(1-s, \operatorname{sym}^2 f)$. \newline
  
   We derive from \eqref{Lambdapkrel}, \eqref{alpha} and \eqref{Lsymexp} that
\begin{align}
\label{alphalambdarel}
\begin{split}
  \alpha_p+\beta_p= &\lambda(p), \\
  \alpha^2_p+\beta^2_p=& \lambda^2(p)-2=\lambda(p^2)-1.
\end{split}
\end{align}

Thus, it follows from the above that
\begin{align}
\label{sumlambdapsquare}
\begin{split}
  \lambda^2(p)=\lambda(p^2)+1.
\end{split}
\end{align}

  We also note that R. A. Rankin \cite{Rankin39} and A. Selberg \cite{Selberg40} showed independently that for $x \geq 1$, 
\begin{align}
\label{lambdasquareasymp}
\begin{split}
  \sum_{n \leq x}\lambda^2(n) =\frac {L(1, \operatorname{sym}^2 f)}{\zeta(2)}x+O(x^{3/5}).
\end{split}
\end{align}

 We end this section by including a result on certain sums over primes.
\begin{lemma}
\label{RS}
 Let $x \geq 2$. We have, for some constant $b_1, b_2$,
\begin{align}
\label{merten}
\sum_{p\le x} \frac{1}{p} =& \log \log x + b_1+ O\Big(\frac{1}{\log x}\Big), \; \mbox{and} \\
\label{merten1}
\sum_{p\le x} \frac{\lambda^2(p)}{p} =& \log \log x + b_2+ O\Big(\frac{1}{\log x}\Big).
\end{align}
\end{lemma}
\begin{proof}
  The expression in \eqref{merten} can be found in part (d) of \cite[Theorem 2.7]{MVa1} and the expression in \eqref{merten1} follows from \cite[Lemma 2.1]{GHH}.
\end{proof}

\subsection{A smooth partition of unity}\label{subsecpartition}

  We include in this section a result concerning a smooth partition of unity, taken from \cite[Approximation Result 1]{Harper23}.
\begin{lemma}\label{apres}
Let $N \in \N$ be large, and $\delta > 0$ be small. There exist functions $g : \mr \rightarrow \mr$ (depending on $\delta$) and $g_{N+1} : \mr \rightarrow \mr$ (depending on $\delta$ and $N$) such that, if we define $g_{j}(x) = g(x - j)$ for all integers $|j| \leq N$, we have the following properties:
\begin{enumerate}
\item $\sum_{|j| \leq N} g_{j}(x) + g_{N+1}(x) = 1$ for all $x \in \mr$;

\item $g(x) \geq 0$ for all $x \in \mr$, and $g(x) \leq \delta$ whenever $|x| > 1$;

\item $g_{N+1}(x) \geq 0$ for all $x \in \mr$, and $g_{N+1}(x) \leq \delta$ whenever $|x| \leq N$;

\item for all $l \in \N$ and all $x \in \mr$, we have the derivative estimate $|\frac{d^{l}}{dx^{l}} g(x)| \leq \frac{1}{\pi (l+1)} (\frac{2\pi}{\delta})^{l+1}$.

\end{enumerate}
\end{lemma}

\subsection{Mean Value Estimations} 
 
 Throughout the paper, unless mentioning otherwise, we denote $(h(p))_{p \text{ prime}}$ a sequence of independent random variables distributed uniformly on the unit circle in $\comc$ and we denote $h(n)$  a Steinhaus random multiplicative function such that $h(n):=\prod_{p^a \| n}h(p)^a$ for all natural numbers $n$.  Therefore, $h$ is a random function taking values in the complex unit circle and completely multiplicative. We denote the expectation by $\mathbb{E}$, the probability measure by $\p$ and the indicator function by $\textbf{1}$. Let $\lfloor x \rfloor$ denote the largest integer not exceeding $x$ and $\lceil x \rceil$ denote the smallest integer exceeding $x$. We also denote throughout the paper by $\Echar$ the averaging over all Dirichlet characters mod $q$, so that for any function $f(\chi)$, 
   $$\Echar f := \frac{1}{q-1} \sum_{\chi \; \text{mod} \; q} f(\chi).$$ 
  
   In this section, we gather some mean value estimations concerning sums or Euler products involving with either Dirichlet characters or random variables. 
 Our first result is just \cite[Lemma 1]{Harper23}.
\begin{lemma}
\label{evenmoment}
   Let $x \geq 1$, and let $(c(n))_{n \leq x}$ be any complex numbers. Let $\mathcal{P}$ be any finite set of primes, let $\mathcal{Q}$ be any non-empty set consisting of some elements of $\mathcal{P}$ and squares of elements of $\mathcal{P}$, and write $U := \max\{m \in \mathcal{Q}\}$ . Finally, let $Q(\chi) := \sum_{m \in \mathcal{Q}} \frac{a(m) \chi(m)}{\sqrt{m}}$, where the $a(m)$ are any complex numbers.

Then for any natural number $l$ such that $x U^l < q$, we have
$$ \Echar \Biggl|\sum_{n \leq x} c(n) \chi(n) \Biggr|^{2} |Q(\chi)|^{2l} \ll \Biggl(\sum_{n \leq x} \tilde{d}(n) |c(n)|^2 \Biggr) \cdot (l !) \Biggl( 2 \sum_{m \in \mathcal{Q}} \frac{v_m |a(m)|^2}{m} \Biggr)^{l} , $$
where $\tilde{d}(n) := \sum_{d|n} \textbf{1}_{p|d \Rightarrow p \in \mathcal{P}}$, and $v_m$ is $1$ if $m$ is a prime and $6$ if $m$ is the square of a prime.
\end{lemma}

 Now, Lemma \ref{evenmoment} allows us to deduce the following result, which is an extension of \cite[Proposition 1]{Harper23}.
\begin{prop}
\label{condexpprop}
Let the functions $g_j, \tilde{d}(n)$ and the associated parameters $N, \delta$, be as in Lemma \ref{apres}. Let $h(n)$ denote a Steinhaus random multiplicative function. Suppose that $x \geq 1$, and let $(c(n))_{n \leq x}$ be any complex numbers such that 
\begin{align}
\label{cncond}
 \sum_{n \leq x} \tilde{d}(n) |c(n)|^2 \ll x\log P.
\end{align} 
 Let $P$ be large, and let $Y \in \N$. Let also $c_0>0$ be a constant. Let $(a_{j}(p^{l}))_{p \leq P}, 1 \leq j \leq Y, l=1,2$ be any sequence of complex numbers such that $|a_{j}(p^{l})| \leq c_0$ for all  $1 \leq j \leq Y, l=1,2$.  Suppose that $x P^{400(c^2_0+1)(Y/\delta)^2 \log(N\log P)} < q$, then we have for any indices $-N \leq j(1), j(2), ..., j(Y) \leq N+1$, 
\begin{align}
\label{Erel}
\begin{split}
 & \Echar \prod_{i=1}^{Y} g_{j(i)}(\Re(\sum_{p \leq P} \frac{a_{i}(p) \chi(p)}{\sqrt{p}} + \frac{a_{i}(p^2) \chi(p^2)}{p})) \Biggl|\sum_{n \leq x} c(n) \chi(n) \Biggr|^2  \\
= & \E \prod_{i=1}^{Y} g_{j(i)}(\Re(\sum_{p \leq P} \frac{a_{i}(p) h(p)}{\sqrt{p}} + \frac{a_{i}(p^2) h(p^2)}{p})) \Biggl|\sum_{n \leq x} c(n) h(n) \Biggr|^2 + O\left(\frac{x}{(N \log P)^{Y(1/\delta)^2}} \right).
\end{split}
\end{align} 
\end{prop}
\begin{proof}
Our proof is a variation of the proof of \cite[Proposition 1]{Harper23}. As shown there, we may apply property (4) of Lemma \ref{apres} to write $g_{j}(x) = \tilde{g}_{j}(x) + r_{j}(x)$ such that $\tilde{g}_{j}(\cdot)$ is a polynomial of degree $2S-1$ for some parameter $S$, and that $|r_{j}(x)| \ll \frac{N |2\pi x/\delta|^{2S}}{\delta S (2S)!}$. It is also shown there that when $x P^{4SY} < q$, we have
\begin{align}
\label{Maintermmatch}
\begin{split}
& \Echar \prod_{i=1}^{Y} \tilde{g}_{j(i)}(\Re(\sum_{p \leq P} \frac{a_{i}(p) \chi(p)}{\sqrt{p}} + \frac{a_{i}(p^2) \chi(p^2)}{p})) \Biggl|\sum_{n \leq x} c(n) \chi(n) \Biggr|^2  \\
= & \E \prod_{i=1}^{Y} \tilde{g}_{j(i)}(\Re( \sum_{p \leq P} \frac{a_{i}(p) h(p)}{\sqrt{p}} + \frac{a_{i}(p^2) h(p^2)}{p})) \Biggl|\sum_{n \leq x} c(n) h(n) \Biggr|^2. 
\end{split}
\end{align} 

  Further, upon using Lemma \ref{evenmoment},  it is shown in the proof of \cite[Proposition 1]{Harper23} that the contribution from all of the remainders $r_{j(i)}(\cdot)$ to the left-hand side expression in \eqref{Erel} is
\begin{align}
\label{Remainderbound}
\begin{split}
 \ll & \Biggl(\sum_{n \leq x} \tilde{d}(n) |c(n)|^2 \Biggr) \cdot \sum_{i=1}^{Y} \frac{N}{\delta} \sqrt{\frac{i}{S}} (\frac{e (\pi/\delta)^2 i (2\max_{1 \leq l \leq i}\sum_{p \leq P} (\frac{|a_{l}(p)|^2 }{p} + \frac{6|a_{l}(p^2)|^2}{p^2}))}{S})^S   \\
& \cdot \Biggl( 1 + O\Biggl(\frac{N}{\delta S} (\frac{e (\pi/\delta)^2 i (2\max_{1 \leq l \leq i}\sum_{p \leq P} (\frac{|a_{l}(p)|^2 }{p} + \frac{6|a_{l}(p^2)|^2}{p^2}))}{S})^S \Biggr) \Biggr)^{i-1} \\
\ll & x\log P \cdot \sum_{i=1}^{Y} \frac{N}{\delta} \sqrt{\frac{i}{S}} (\frac{e (\pi/\delta)^2 i (2\max_{1 \leq l \leq i}\sum_{p \leq P} (\frac{|a_{l}(p)|^2 }{p} + \frac{6|a_{l}(p^2)|^2}{p^2}))}{S})^S   \\
& \cdot \Biggl( 1 + O\Biggl(\frac{N}{\delta S} (\frac{e (\pi/\delta)^2 i (2\max_{1 \leq l \leq i}\sum_{p \leq P} (\frac{|a_{l}(p)|^2 }{p} + \frac{6|a_{l}(p^2)|^2}{p^2}))}{S})^S \Biggr) \Biggr)^{i-1},
\end{split}
\end{align} 
  where the last estimation above follows from \eqref{cncond}.

 Note also that as $|a_{i}(p)|, |a_{i}(p^2)| \leq c_0$, it follows from \eqref{merten} that
$$2\max_{1 \leq l \leq i} \sum_{p \leq P} (\frac{|a_{i}(p)|^2 }{p} + \frac{6|a_{i}(p^2)|^2}{p^2}) \leq 2c^2_0\sum_{p \leq P} \frac{1}{p} + O(1) = 2c^2_0\log\log P + O(1).$$ 

  We deduce from \eqref{Remainderbound} and the above that the contribution from all of the remainders $r_{j(i)}(\cdot)$ to the left-hand side expression in \eqref{Erel} is
\begin{align}
\label{Remainderbound1}
\begin{split}
 \ll & x\log P \cdot \sum_{i=1}^{Y} \frac{N}{\delta} \sqrt{\frac{i}{S}} (\frac{e (\pi/\delta)^2 i (2c^2_0\log\log P + O(1))}{S})^S  \cdot \Biggl( 1 + O\Biggl(\frac{N}{\delta S} (\frac{e (\pi/\delta)^2 i (2c^2_0\log\log P + O(1))}{S})^S \Biggr) \Biggr)^{i-1} . 
\end{split}
\end{align} 

Now we set $S = 100c^2_0Y \lfloor (1/\delta)^2 \log(N\log P) \rfloor$ to see that the condition $x P^{4SY} < q$ is satisfied by our assumption that $x P^{400(c^2_0+1)(Y/\delta)^2 \log(N\log P)}  < q$. Moreover, the expression in \eqref{Remainderbound1} is
$$ \ll x\log P \cdot \sum_{i=1}^{Y} \frac{N}{\delta} \sqrt{\frac{i}{S}} 0.6^S \left( 1 + O(\frac{N}{\delta S} 0.6^S ) \right)^{i-1} \ll x\log P \cdot NY \cdot 0.6^S \ll \frac{x}{(N \log P)^{Y(1/\delta)^2}}.$$

  Note that as pointed out in the proof of \cite[Proposition 1]{Harper23}, one has the same overall bound for the contribution from the remainders $r_{j(i)}(\cdot)$ to the first expression on the right-hand side of \eqref{Erel}. This together with \eqref{Maintermmatch} now leads to the desired expression in \eqref{Erel} and therefore completes the proof of the proposition.
\end{proof}

 We now set $x=1$ and $c(1)=1$ in Proposition \ref{condexpprop} to deduce the following special case of it. 
\begin{prop}\label{boxprobprop}
 Let the functions $g_j$ and the associated parameters $N, \delta$, be as in Lemma \ref{apres}. Let $h(n)$ denote a Steinhaus random multiplicative function.
Let $P$ be large, and let $Y \in \N$. Let also $c_0>0$ be a constant. Let $(a_{j}(p^{l}))_{p \leq P}, 1 \leq j \leq Y, l=1,2$ be any sequence of complex numbers such that $|a_{j}(p^{l})| \leq c_0$ for all  $1 \leq j \leq Y, l=1,2$.  Suppose that $x P^{400(c^2_0+1)(Y/\delta)^2 \log(N\log P)} < q$, then we have for any indices $-N \leq j(1), j(2), ..., j(Y) \leq N+1$, 
\begin{eqnarray}
\Echar \prod_{i=1}^{Y} g_{j(i)}(\Re(\sum_{p \leq P} \frac{a_{i}(p) \chi(p)}{\sqrt{p}} + \frac{a_{i}(p^2) \chi(p^2)}{p})) & = & \E \prod_{i=1}^{Y} g_{j(i)}(\Re(\sum_{p \leq P} \frac{a_{i}(p) h(p)}{\sqrt{p}} + \frac{a_{i}(p^2) h(p^2)}{p})) + \nonumber \\
&& + O\left(\frac{1}{(N \log P)^{Y(1/\delta)^2}} \right) . \nonumber
\end{eqnarray}
\end{prop}

  Our next result treats the expectation of certain random Euler products.
\begin{lemma}
\label{eulerproduct}
    Let $h(n)$ be a Steinhaus random multiplicative function, $\alpha,\beta,\sigma_1,\sigma_2\geq 0$ and $t_1,t_2 \in \rear$.  Suppose that $100(1+\max(\alpha^2,\beta^2))\leq z<y$. Then
\begin{align}
\label{Eest0}
\begin{split}
        &  \mathbb{E}\prod_{z\leq p\leq y} \Big|1-\frac{\alpha_ph(p)}{p^{1/2+\sigma_1+it_1}}\Big|^{-2\alpha}\Big|1-\frac{\beta_ph(p)}{p^{1/2+\sigma_1+it_1}}\Big|^{-2\alpha}
\Big|1-\frac{\alpha_ph(p)}{p^{1/2+\sigma_2+it_2}}\Big|^{-2\beta}\Big|1-\frac{\beta_ph(p)}{p^{1/2+\sigma_2+it_2}}\Big|^{-2\beta} \\
        =&\exp\bigg(\sum_{p\leq y}\frac{\alpha^2\lambda^2(p)}{p^{1+2\sigma_1} }+\frac{\beta^2\lambda^2(p)}{p^{1+2\sigma_2} }+\frac{2\alpha\beta \lambda^2(p)\cos((t_2-t_1)\log p)}{p^{1+\sigma_1+\sigma_2 } }+O\Big(\frac{\max(\alpha,\alpha^3,\beta,\beta^3)}{z^{1/2}}\Big)\bigg).
\end{split}
\end{align}

   Moreover, we have for any real $t$ and $u$, any real $400(1 + u^2) \leq x \leq y$, and any real $\sigma \geq - 1/\log y$, 
\begin{align}
\label{Eest}
\begin{split}
 & \E \prod_{x < p \leq y} \Big|1-\frac{\alpha_ph(p)}{p^{1/2+\sigma}}\Big|^{-2}\Big|1-\frac{\beta_ph(p)}{p^{1/2+\sigma}}\Big|^{-2}
\Big|1-\frac{\alpha_ph(p)}{p^{1/2+\sigma+it}}\Big|^{-iu}\Big|1-\frac{\beta_ph(p)}{p^{1/2+\sigma+it}}\Big|^{-iu}\\
=&  \exp\Big(\sum_{x < p \leq y} \frac{1 + iu\cos(t\log p) - u^{2}/4}{p^{1 + 2\sigma}} + T(u)\Big), 
\end{split}
\end{align}
where $T(u) = T_{x,y,\sigma,t}(u)$ satisfies $|T(u)| \ll \frac{1 + |u|^3}{\sqrt{x} \log x}$ for any $u$, and $|T'(u)| \ll \frac{1}{\sqrt{x} \log x}$ for $|u| \leq 1$.
\end{lemma}
\begin{proof}
  The expression given in \eqref{Eest0} is established in \cite[Lemma 2.7]{G&Zhao25-11}. The proof of \eqref{Eest} is similar and follows by an adaption of the proof of \cite[Lemma 1]{Harper20}. We set 
$$R_{p}(t) := -\Re\log(1 - \frac{\alpha_ph(p)}{p^{1/2+\sigma+it}})-\Re\log(1 - \frac{\beta_ph(p)}{p^{1/2+\sigma+it}})$$
 to see that 
\begin{align*}
%%\label{Eest}
\begin{split}
& \Big|1-\frac{\alpha_ph(p)}{p^{1/2+\sigma}}\Big|^{-2}\Big|1-\frac{\beta_ph(p)}{p^{1/2+\sigma}}\Big|^{-2}
\Big|1-\frac{\alpha_ph(p)}{p^{1/2+\sigma+it}}\Big|^{-iu}\Big|1-\frac{\beta_ph(p)}{p^{1/2+\sigma+it}}\Big|^{-iu} \\
 =& \exp\Big(2 R_{p}(0) +iu R_{p}(t) \Big) = 1 + \sum_{j=1}^{\infty} \frac{(2 R_{p}(0) +iu R_{p}(t))^j}{j!} . 
\end{split}
\end{align*}

 By the Taylor expansion, \eqref{alpha} and \eqref{alphalambdarel},  we see that 
\begin{align*}
%%\label{Eest}
\begin{split}
& R_{p}(t) = \sum_{j=1}^{\infty} \frac{(\alpha^j_p+\beta^j_p) \Re(h(p)p^{-it})^j}{j p^{j(1/2+\sigma)}} = \frac{\lambda(p)\Re h(p) p^{-it}}{p^{1/2+\sigma}} + O(\frac{1}{p^{1+2\sigma}}).
\end{split}
\end{align*}
 It follows that
\begin{align*}
%%\label{Eest}
\begin{split}
\E R_{p}(t)^2 =&  \E\frac{(\lambda(p)\Re h(p)p^{-it})^2}{p^{1 + 2\sigma}} + O(\frac{1}{p^{3/2+3\sigma}}) = \frac{\lambda^2(p)}{2 p^{1 + 2\sigma}} + O(\frac{1}{p^{3/2+3\sigma}}), \\
\E R_{p}(0) R_{p}(t) = & \E\frac{(\lambda(p)\Re h(p))(\lambda(p)\Re h(p)p^{-it})}{p^{1 + 2\sigma}} + O(\frac{1}{p^{3/2+3\sigma}}) = \frac{\lambda^2(p)\cos(t\log p)}{2 p^{1 + 2\sigma}} + O(\frac{1}{p^{3/2+3\sigma}}) .
\end{split}
\end{align*}
 Note also that by \eqref{alpha}, we have $|R_{p}(t)^j| \leq (\sum_{m=1}^{\infty} \frac{2}{mp^{m(1/2+\sigma)}})^j = \frac{2^j}{(p^{1/2+\sigma}-1)^j}$ for $j \geq 3$.

We then proceed as in the proof of \cite[Lemma 1]{Harper20} to see that the expression in \eqref{Eest} is valid. This completes the proof of the lemma. 
\end{proof}

 Note that as a special case of \eqref{Eest} and by a similar argument, we have for any $400 \leq x \leq y$ and any $\sigma \geq -1/\log y$, 
\begin{align}
\label{Eprodsquare}
\begin{split}
\E \prod_{x < p \leq y} \Big|1-\frac{\alpha_ph(p)}{p^{1/2+\sigma}}\Big|^{-2}\Big|1-\frac{\beta_ph(p)}{p^{1/2+\sigma}}\Big|^{-2} = & \exp\Big (\sum_{x < p \leq y} \frac{\lambda^2(p)}{p^{1 + 2\sigma}} + O(\frac{1}{\sqrt{x}\log x}) \Big ),  \\
\E \prod_{x < p \leq y} \Big|1-\frac{\alpha_ph(p)}{p^{1/2+\sigma}}\Big|^{2}\Big|1-\frac{\beta_ph(p)}{p^{1/2+\sigma}}\Big|^{2} = & \exp\Big (\sum_{x < p \leq y} \frac{\lambda^2(p)}{p^{1 + 2\sigma}} + O(\frac{1}{\sqrt{x}\log x}) \Big ). 
\end{split}
\end{align}
  
The following lemma is taken from \cite[Theorem 5.4]{MVa1}, which gives a version of Parseval’s identity for Dirichlet series.
\begin{lemma}
\label{parseval}
    Let $(a_n)_{n\geq 1}$ be a sequence of complex numbers and $F(s)=\sum_{n=1}^{\infty} a_nn^{-s}$ be the corresponding Dirichlet series. If $\sigma_c$ denotes its abscissa of convergence, then, for any $\sigma>\max(0,\sigma_c)$, we have
    \begin{equation*}
        \int\limits_{1}^{\infty} \frac{\big|\sum_{n\leq x}a_n\big|^2 }{x^{1+2\sigma }} \dif x=\frac{1}{2\pi}\int\limits_{-\infty}^{+\infty}\frac{|F(\sigma+it)|^2}{|\sigma+it|^2} \dif t.
    \end{equation*}
\end{lemma}

\subsection{Probabilistic Evaluations}
\label{secprobcalc}

  In this section, we give more results concerning Steinhaus random multiplicative functions. These results ultimately lead to Proposition \ref{propintrandfcnmain} below on an upper bound for the low moments of a short integral of a random Euler product, which will play a crucial role in our proof of Theorem \ref{lowerboundsfixedmodmean}. Our treatments in this section largely follow those given in \cite{Harper20}. 

Let $x$ be large and $-1/100 \leq \sigma \leq 1/100$. Recall that  $h(n)$ is a Steinhaus random multiplicative function. As in \cite[Section 3.2]{Harper20}, let $\tilde{\p} = \tilde{\p}_{x,\sigma}$ be a new probability measure such that for each event $A$, 
$$ \tilde{\p}(A) := \frac{\E \textbf{1}_{A} \prod_{p \leq x^{1/e}} \left|1 - \frac{\alpha_p h(p)}{p^{1/2+\sigma}}\right|^{-2}\left|1 - \frac{\beta_p h(p)}{p^{1/2+\sigma}}\right|^{-2}}{\E \prod_{p \leq x^{1/e}} \left|1 - \frac{\alpha_p h(p)}{p^{1/2+\sigma}}\right|^{-2}\left|1 - \frac{\beta_p h(p)}{p^{1/2+\sigma}} \right|^{-2}}.$$
We also denote $\tilde{\E}$ the expectation with respect to the measure $\tilde{\p}$ such that for any random variable $F$,
$$ \tilde{\E}F := \frac{\E F \prod_{p \leq x^{1/e}} \left|1 - \frac{\alpha_p h(p)}{p^{1/2+\sigma}}\right|^{-2}\left|1 - \frac{\beta_p h(p)}{p^{1/2+\sigma}}\right|^{-2}}{\E \prod_{p \leq x^{1/e}} \left|1 - \frac{\alpha_p h(p)}{p^{1/2+\sigma}}\right|^{-2}\left|1 - \frac{\beta_p h(p)}{p^{1/2+\sigma}} \right|^{-2}}, $$

We further set for each $l \in \N \cup \{0\}$, 
\begin{align}
\label{Ildef}
I_l(s) := \prod_{x^{e^{-(l+2)}} < p \leq x^{e^{-(l+1)}}} (1 - \frac{\alpha_ph(p)}{p^s})^{-1}(1 - \frac{\beta_ph(p)}{p^s})^{-1}.
\end{align}
 Let also $(l_j)_{j=1}^{n}$ be a strictly decreasing sequence of non-negative integers such that $l_1 \leq \log\log x - 2$, and let $(x_j)_{j=1}^{n}$  be a corresponding increasing sequence of real numbers such that $x_j := x^{e^{-(l_j + 1)}}$. 
  
   We have the following result that is analogue to \cite[Lemma 3]{Harper20}.
\begin{lemma}
\label{lem3}
With the notation as above. Suppose that $x_1$ is sufficiently large and that $|\sigma| \leq 1/\log x_n$ and that $(v_j)_{j=1}^{n}$ is any sequence of real numbers satisfying for any $1 \leq j \leq n$, 
 $$ |v_j| \leq (1/40)\sqrt{\log x_j} + 2. $$

Then we have for any sequence of real numbers $(t_j)_{j=1}^{n}$,  
\begin{align*}
%%\label{lambdansquarenlargesimplified}
& \tilde{\p}(v_j \leq \log|I_{l_j}(1/2 + \sigma + it_j)| \leq v_j + 1/j^2 \; \forall 1 \leq j \leq n)  =  \left(1+O\left(\frac{1}{x_{1}^{1/100}}\right) \right) \p(v_j \leq N_j \leq v_j + 1/j^2 \; \forall 1 \leq j \leq n) , 
\end{align*}
where $N_j$ are independent Gaussian random variables with mean $\sum_{x_{j}^{1/e} < p \leq x_j} \frac{\lambda^2(p)\cos(t_j \log p)}{p^{1 + 2\sigma}}$ and variance $\sum_{x_{j}^{1/e} < p \leq x_j} \frac{\lambda^2(p)}{2p^{1 + 2\sigma}}$.
\end{lemma}
\begin{proof}
  The proof follows by a modification of the proof of \cite[Lemma 3]{Harper20}. We first note that by \eqref{merten1}, we have for all $1 \leq j \leq n$,  
$$ e^{-2}(1 + O(\frac{1}{\log x_j})) \leq \sum_{x_{j}^{1/e} < p \leq x_j} \frac{\lambda^2(p)}{p^{1 + 2\sigma}} \leq e^{2} \sum_{x_{j}^{1/e} < p \leq x_j} \frac{\lambda^2(p)}{p} = e^{2}(1 + O(\frac{1}{\log x_j})). $$
  Moreover, we have
\begin{align*}
%%\label{lambdansquarenlargesimplified}
\E e^{iuN_j} = & \E \exp\{iu\sum_{x_{j}^{1/e} < p \leq x_j} \frac{\lambda^2(p)\cos(t_j \log p)}{p^{1 + 2\sigma}} + iu\sqrt{\sum_{x_{j}^{1/e} < p \leq x_j} \frac{\lambda^2(p)}{2p^{1 + 2\sigma}}} N(0,1)\} \nonumber \\
= & \exp\{\sum_{x_{j}^{1/e} < p \leq x_j} \frac{\lambda^2(p)(iu\cos(t_j \log p) - u^{2}/4)}{p^{1 + 2\sigma}} \},
\end{align*}
 where $N(0,1)$ denotes a normal random variable with mean $0$ and variance $1$. 
  The assertion of the lemma now follows from the above upon using \eqref{Eest} and arguing as in the proof of \cite[Lemma 3]{Harper20}. This completes the proof of the lemma. 
\end{proof}

 Now Lemma \ref{lem3} allows us to establish the following analogue of \cite[Lemma 4]{Harper20}.
\begin{lemma}
\label{lem4}
 With the notation as in Lemma \ref{lem3}. Suppose $(u_j)_{j=1}^{n}$ and $(v_j)_{j=1}^{n}$ are sequences of real numbers such that for any $1 \leq j \leq n$, 
$$ -(1/80)\sqrt{\log x_j} \leq u_j \leq v_j \leq (1/80)\sqrt{\log x_j}. $$
Then we have
\begin{align*}
%%\label{lambdansquarenlargesimplified}
& \p(u_j + 2 \leq \sum_{m=1}^{j} N_m \leq v_j - 2 \; \forall j \leq n) \\
 \ll & \ \tilde{\p}(u_j \leq \sum_{m=1}^{j} \log|I_{l_m}(\frac{1}
{2} + \sigma + it_m)| \leq v_j \; \forall j \leq n) \\
\ll & \ \p(u_j - 2 \leq \sum_{m=1}^{j} N_m \leq v_j + 2 \; \forall j \leq n). 
\end{align*}

 In addition, if the numbers $(t_j)_{j=1}^{n}$ satisfy $|t_j| \leq \frac{1}{j^{2/3} \log x_j}$ then we have
\begin{eqnarray}
&& \p((u_j - j) + O(1) \leq \sum_{m=1}^{j} G_m \leq (v_j-j) - O(1) \; \forall j \leq n) \nonumber \\
& \ll & \tilde{\p}(u_j \leq \sum_{m=1}^{j} \log|I_{l_m}(1/2 + \sigma + it_m)| \leq v_j \; \forall j \leq n) \nonumber \\
& \ll & \p((u_j - j) - O(1) \leq \sum_{m=1}^{j} G_m \leq (v_j-j) + O(1) \; \forall j \leq n) , \nonumber
\end{eqnarray}
where $G_m$ are independent Gaussian random variables, each having mean 0 and variance $\sum_{x_{m}^{1/e} < p \leq x_m} \frac{\lambda^2(p)}{2p^{1 + 2\sigma}}$.
\end{lemma}
\begin{proof}
 The assertion of the lemma follows from a modification of the proof of \cite[Lemma 4]{Harper20}, upon using \eqref{merten1} to see that under our conditions that $|\sigma| \leq 1/\log x_n$ and $|t_m| \leq \frac{1}{m^{2/3}\log x_m}$,
\begin{align*}
%%\label{lambdansquarenlargesimplified}
\begin{split}
\sum_{x_{m}^{1/e} < p \leq x_m} \frac{\lambda^2(p)\cos(t_m \log p)}{p^{1 + 2\sigma}} = & \sum_{x_{m}^{1/e} < p \leq x_m} \frac{\lambda^2(p)}{p^{1 + 2\sigma}} + O(\sum_{x_{m}^{1/e} < p \leq x_m} \frac{\lambda^2(p)(|t_m|\log p)^2}{p^{1 + 2\sigma}}) \\
= & \sum_{x_{m}^{1/e} < p \leq x_m} \frac{\lambda^2(p)}{p} + O(\sum_{x_{m}^{1/e} < p \leq x_m} \frac{\lambda^2(p)(|\sigma| \log p + (|t_m|\log p)^2)}{p}) \nonumber \\
 = & 1 + O(\frac{1}{\log x_m} + \frac{\log x_m}{\log x_n} + \frac{1}{m^{4/3}}) \\
=& 1 + O(\frac{1}{e^{n-m}} + \frac{1}{m^{4/3}}). 
\end{split}
\end{align*}
 This completes the proof of the lemma.
\end{proof}

  Now, using Lemma \ref{lem3}--\ref{lem4} with Probability Results 1 and 2 in \cite{Harper20}, we arrive at the following result analogue to \cite[Proposition 5]{Harper20}.
\begin{prop}
\label{prop5}
There exists a large natural number $B$ so that the following is valid. Let $(l_j)_{j=1}^{n}$ be a decreasing sequence of non-negative integers defined by $l_j := \lfloor \log\log x \rfloor - (B+1) - j$, where $n \leq \log\log x - (B+1)$ is large. Suppose that $|\sigma| \leq \frac{1}{e^{B+n+1}}$, and that $(t_j)_{j=1}^{n}$ is a sequence of real numbers satisfying $|t_j| \leq \frac{1}{j^{2/3} e^{B+j+1}}$ for all $j$.

Then with $I_{l}(s)$ being defined as in \eqref{Ildef}, we have uniformly for any large $a$ and any function $g(n)$ satisfying $|g(n)| \leq 10\log n$, 
$$ \tilde{\p}(-a -Bj \leq \sum_{m=1}^{j} \log|I_{l_m}(1/2 + \sigma + it_m)| \leq a + j + g(j) \; \forall j \leq n) \asymp \min\{1,\frac{a}{\sqrt{n}}\} . $$
\end{prop}
\begin{proof}
  This follows from a straightforward modification of the proof of \cite[Proposition 5]{Harper20}, with $G_j$ there being replaced by independent Gaussians having mean $0$ and variance $\sum_{x_{j}^{1/e} < p \leq x_j} \frac{\lambda^2(p)}{2p^{1 + 2\sigma}}$. Here we notice that by \eqref{merten1}, the variance is $\leq \frac{e^{(2\log x_j)/\log x_n}}{2} \sum_{x_{j}^{1/e} < p \leq x_j} \frac{\lambda^2(p)}{p}  \leq 5$, and similarly the variance is $\geq 1/20$. Thus, one is able to apply Probability Results 1 and 2 in \cite{Harper20} to deduce the assertion of the proposition. This completes the proof. 
\end{proof}

  Now, for each $|t| \leq 1/2$ and integers $0 \leq j \leq \log\log x - 2$, we define iteratively $t(j)$ by setting $t(-1) = t$ and 
$$ t(j) := \max\{u \leq t(j-1): u = \frac{n}{((\log x)/e^{j+1}) \log((\log x)/e^{j+1})} \; \text{for some} \; n \in \mz \} . $$
  let $B$ be the large fixed natural number from Proposition \ref{prop5}, and let $\mathcal{G}(m)$ denote the event that for all $|t| \leq 1/2$ and all $m \leq j \leq \log\log x - B - 2$, we have
$$ \Biggl( \frac{\log x}{e^{j+1}} e^{g(x,j)} \Biggr)^{-1} \leq \prod_{l = j}^{\lfloor \log\log x \rfloor - B - 2} |I_{l}(1/2 - \frac{m}{\log x} + it(l))| \leq \frac{\log x}{e^{j+1}} e^{g(x,j)} , $$
where $g(x,j) :=  C\min\{\sqrt{\log\log x}, \frac{1}{1-k} \} + 2\log\log(\frac{\log x}{e^{j+1}})$ for a large constant $C$.

   For any large quantity $P$ and any complex $s$ with $\Re(s) > 0$, let $F^{\text{rand}}_{P, f}$ denote the partial Euler product of $h(n)\lambda(n)$ over $P$-smooth numbers, so that we have 
\begin{align}
\label{FPdef}
\begin{split}
F^{\text{rand}}_{P,f}(s) = \sum_{\substack{n=1, \\ n \; \text{is} \; P \; \text{smooth}}}^{\infty} \frac{h(n)\lambda(n)}{n^s}= \prod_{p \leq P} \left(1 - \frac{\alpha_ph(p)}{p^s}\right)^{-1}\left(1 - \frac{\beta_ph(p)}{p^s}\right)^{-1}, 
\end{split}
\end{align}
  where the last equality above follows from \eqref{Lambdapkrel}. 

For any integer $0 \leq m \leq \lfloor \log\log x \rfloor$, we denote $F_m$ for brevity the function $F^{\text{rand}}_{x^{e^{-(m+1)}},f}(s)$. With the above notation, we observe by \eqref{merten1} that we have,
\begin{align}
\label{EFsquare}
\begin{split}
\E |F_{m}(1/2 - \frac{m}{\log x}+it)|^2  = & \exp\Big (\sum_{p \leq x^{e^{-(k+1)}}} \frac{\lambda^2(p)}{p^{1 - 2m/\log x}} + O(1) \Big )  \\
= & \exp\{\sum_{p \leq x^{e^{-(m+1)}}} \frac{\lambda^2(p)}{p} + O(\sum_{p \leq x^{e^{-(m+1)}}} \frac{m \lambda^2(p) \log p}{p \log x} + 1) \} \ll \frac{\log x}{e^m}.
\end{split}
\end{align}
   We further  make use of \eqref{Eprodsquare} and follow the proofs of Key Propositions 1 and 2 in \cite{Harper20} to derive the following two results.
\begin{prop}
With the natation as above. For all large $x$, and uniformly for $0 \leq m \leq \lfloor \log\log\log x \rfloor$ and $2/3 \leq k \leq 1$, we have
$$ \E( \textbf{1}_{\mathcal{G}(m)} \int_{-1/2}^{1/2} |F_{m}(1/2 - \frac{m}{\log x} + it)|^2 dt )^k \ll \Biggl( \frac{\log x}{e^{m}} C \min\{1, \frac{1}{(1-k)\sqrt{\log\log x}} \} \Biggr)^k.$$
\end{prop}

\begin{prop}
With the natation as above. For all large $x$, and uniformly for $0 \leq m \leq \lfloor \log\log\log x \rfloor$ and $2/3 \leq k \leq 1$, we have
$$ \p(\mathcal{G}(m) \; \text{fails}) \ll e^{-2C \min\{\sqrt{\log\log x}, \frac{1}{1-k} \}}. $$
\end{prop}

   As a consequence of the above two propositions, we further follow the proof of the upper bound in \cite[Theorem 1]{Harper20} given in Section 4.1 of \cite{Harper20} to obtain the next result concerning the $k$-th moment of a short integral of $F_m$ for certain restricted $k \leq 1$. 
\begin{prop}
\label{propintrandfcn}
 With the natation as above. We have uniformly for all $0 \leq m \leq \lfloor \log\log\log x \rfloor$ and $2/3 \leq k \leq 1 - \frac{1}{\sqrt{\log\log x}}$,
\begin{align}
\label{Erandfcnupperbound}
  \E(\int_{-1/2}^{1/2} |F_{m}(1/2 - \frac{m}{\log x} + it)|^2 dt )^k \ll (\frac {\log x}{e^{m} (1-k) \sqrt{\log\log x}})^k. 
\end{align} 
\end{prop}

   With the aid of the above proposition, we now extend the upper bound for the left-hand side expression in \eqref{Erandfcnupperbound} to all $k \leq 1$.
\begin{prop}
\label{propintrandfcnmain}
 With the natation as above. We have uniformly for all $0 \leq m \leq \lfloor \log\log\log x \rfloor$ and $2/3 \leq k \leq 1$,
\begin{align}
\label{Erandfcnupperboundmain}
  \E(\int_{-1/2}^{1/2} |F_{m}(1/2 - \frac{m}{\log x} + it)|^2 dt )^k \ll (\frac {\log x}{1+(1-k) \sqrt{\log\log x}})^k. 
\end{align} 
\end{prop}
\begin{proof}
  Note that by H\"{o}lder's inequality and \eqref{EFsquare}, we have
\begin{align*}
%%\label{Erandfcnupperboundmain2}
\begin{split}
  \E(\int_{-1/2}^{1/2} |F_{m}(1/2 - \frac{m}{\log x} + it)|^2 dt )^k \ll & \Big (\int_{-1/2}^{1/2} \E|F_{m}(1/2 - \frac{m}{\log x} + it)|^2 dt \Big)^k \ll (\log x)^k.
\end{split}
\end{align*} 
  This implies the estimation given in \eqref{Erandfcnupperboundmain} for $1 - \frac{1}{\sqrt{\log\log x}}<k \leq 1$. As the estimation given in \eqref{Erandfcnupperboundmain} follows from \eqref{Erandfcnupperbound} when $2/3 \leq k \leq 1 - \frac{1}{\sqrt{\log\log x}}$, this completest the proof. 
\end{proof}
   
  We set $m=0$ and $P=x^{1/e}$ in the above proposition to deduce a special case of it.
\begin{corollary}
\label{mcres1}
Uniformly for all large $P$ and $2/3 \leq k \leq 1$, we have
$$ \E( \int_{-1/2}^{1/2} |F_{P,f}^{\text{rand}}(1/2 + it)|^2 dt )^k \ll \Biggl(\frac{\log P}{1 + (1-k)\sqrt{\log\log P}}\Biggr)^{k} . $$
\end{corollary}

 We shall also need to apply a discrete version of Corollary \ref{mcres1} in our proof of Theorem \ref{lowerboundsfixedmodmean}. For this, we note the following result. 
\begin{lemma}\label{disccontlem}
For any large $P$, we have
\begin{align}
\label{Ediff}
\begin{split} \E \sum_{|j| \leq \frac{\log^{1.01}P}{2}} \int_{-\frac{1}{2\log^{1.01}P}}^{\frac{1}{2\log^{1.01}P}} |F_{P,f}^{\text{rand}}(1/2 + i\frac{j}{\log^{1.01}P} + it) - F_{P,f}^{\text{rand}}(1/2 + i\frac{j}{\log^{1.01}P})|^2 dt \ll \log^{0.99}P. 
\end{split}
\end{align} 
\end{lemma}
\begin{proof}
  We follow the proof of \cite[Lemma 2]{Harper23} to see that the left-hand side in \eqref{Ediff} is
\begin{eqnarray}
& = & \sum_{|j| \leq \frac{\log^{1.01}P}{2}} \int_{-\frac{1}{2\log^{1.01}P}}^{\frac{1}{2\log^{1.01}P}} \E |\sum_{\substack{n=1, \\  n \; \text{is} \; P \; \text{smooth}}}^{\infty} \frac{h(n)\lambda(n)}{n^{1/2 + i\frac{j}{\log^{1.01}P}}} (n^{-it} - 1)|^2 dt \nonumber \\
& = & \sum_{|j| \leq \frac{\log^{1.01}P}{2}} \int_{-\frac{1}{2\log^{1.01}P}}^{\frac{1}{2\log^{1.01}P}} \sum_{\substack{n=1, \\  n \; \text{is} \; P \; \text{smooth}}}^{\infty} \frac{\lambda^2(n)|n^{-it}-1|^2}{n} dt . \nonumber
\end{eqnarray}
Note that $|n^{-it} - 1| \ll \min (|t|\log n, 1) \leq \min(\frac{\log n}{\log^{1.01}P}, 1)$, so that the contribution to the series from those $n \leq P^{\log\log P}$ is $\ll \frac{(\log\log P)^2}{\log^{0.02}P} \sum_{\substack{n=1, \\  n \; \text{is} \; P \; \text{smooth}}}^{\infty} \frac{\lambda^2(n)}{n}$.
Note that by \eqref{merten1},
\begin{align}
\label{lambdansquare0}
 \sum_{\substack{n=1, \\  n \; \text{is} \; P \; \text{smooth}}}^{\infty} \frac{\lambda^2(n)}{n} =\prod_{p \leq P}(1+\frac {\lambda(p)^{2}}{p}+O(\frac 1{p^{2}}))
=\prod_{p \leq P}(1-\frac {\lambda(p)^{2}}{p})^{-1}\cdot \prod_{p \leq P}(1+O(\frac 1{p^{2}})) \ll \exp (-\log (1-\frac {\lambda(p)^{2}}{p})) \ll \log P.
\end{align}
  It follows that the contribution to the series from those $n \leq P^{\log\log P}$ is $\ll (\log\log P)^2 \log^{0.98}P$. On the other hand, using the trivial bound $|n^{-it} - 1| \ll 1$ and arguing similar to \eqref{lambdansquare0}, we see that the contribution to the series from those $n > P^{\log\log P}$ is
\begin{align}
\label{lambdansquarenlarge}\ll e^{-\log\log P} \sum_{\substack{n=1, \\  n \; \text{is} \; P \; \text{smooth}}}^{\infty} \frac{\lambda^2(n)}{n^{1-1/\log P}} = e^{-\log\log P} \prod_{p \leq P} (1 - \frac{\lambda^2(p)}{p^{1-1/\log P}})^{-1} \ll e^{-\log\log P} \exp (\sum_{p \leq P}\frac {\lambda^2(p)}{p^{1-1/\log P}}).
\end{align}
 Note that when $p \leq P$, we have $e^{\log p/\log P} \ll 1$ so that by \eqref{merten1}, 
\begin{align*}
%%\label{lambdansquare0}
\sum_{p \leq P}\frac {\lambda^2(p)}{p^{1-1/\log P}} \ll \sum_{p \leq P}\frac {\lambda^2(p)}{p}  \ll \log \log P.
\end{align*}
  It follows from this that the expressions in \eqref{lambdansquarenlarge} is
\begin{align*}
%%\label{lambdansquarenlargesimplified}
 \ll e^{-\log\log P}\log P = 1 .
\end{align*}
 This completes the proof of the lemma.
\end{proof}

   We now apply Corollary \ref{mcres1}, Lemma \ref{disccontlem} and arguing as in the proof of Multiplicative Chaos Result 2 in \cite{Harper23} to obtain the following  discrete version of Corollary \ref{mcres1}.
\begin{lemma}\label{mcres2}
Uniformly for all large $P$ and $2/3 \leq k \leq 1$, we have
$$  \E( \frac{1}{\log^{1.01}P} \sum_{|j| \leq (\log^{1.01}P)/2} |F_{P,f}^{\text{rand}}(1/2 + i\frac{j}{\log^{1.01}P})|^2 )^k \ll \Biggl(\frac{\log P}{1 + (1-q)\sqrt{\log\log P}}\Biggr)^{k} . $$
\end{lemma}

\section{Proof of Theorem \ref{lowerboundsfixedmodmean}} 
\label{outline}

\subsection{Initial Treatments}
  \label{thm1small}
  
We may restrict to the range $2/3 \leq k \leq 1$ throughout the proof, since if $k$ is smaller we can apply H\"{o}lder's inequality to bound $\Echar |\sum_{n \leq x} \chi(n)\lambda(n)|^{2k}$ from above by $(\Echar |\sum_{n \leq x} \chi(n)\lambda(n)|^{4/3})^{3k/2}$, and then use the $k=2/3$ case. We may assume $L = \min\{x,q/x\}$ is large throughout the proof for otherwise the assertion of the Theorem holds trivially. We now set a parameter $P$ to be the largest number below $\exp(\log^{1/6}L)$ such that $\log^{0.01}P$ is an integer. It follows that $\log P \asymp \log^{1/6}L$ and $\log\log P \asymp \log\log L$, so that in order to establish Theorem \ref{lowerboundsfixedmodmean},  it suffices to show that
\begin{equation}\label{stpmaindisplay}
\Echar |\sum_{n \leq x} \chi(n)\lambda(n)|^{2k} \ll \Biggl(\frac{x}{1 + (1-k)\sqrt{\log\log P}} \Biggr)^k .
\end{equation}

Set $M := 2\log^{1.02}P$. We also set for each integer $m$ satisfying $|m| \leq M$ and each character $\chi$ mod $q$, 
\begin{align}
\label{Skf}
 S_{m,f}(\chi) := \Re \sum_{p \leq P} (\frac{\chi(p)\lambda(p)}{p^{1/2 + im/\log^{1.01}P}} + \frac{\chi(p)^2(\lambda(p^2)-\lambda(p)^2/2}{p^{1 + 2im/\log^{1.01}P}}).
\end{align}
 
  Let $P(n)$ denote the largest prime factor of $n$. We apply H\"{o}lder's inequality and orthogonality of Dirichlet characters to see that
\begin{align}
\label{EPnsmall} 
 \Echar |\sum_{n \leq x, P(n) \leq x^{1/\log\log x}} \chi(n)\lambda(n)|^{2k} \leq \Biggl(\Echar |\sum_{n \leq x, P(n) \leq x^{1/\log\log x}} \chi(n)\lambda(n)|^{2} \Biggr)^k = \Biggl(\sum_{n \leq x, P(n) \leq x^{1/\log\log x}}\lambda(n)^{2}\Biggr)^k.
\end{align}
We now follow the proof of \cite[Theorem 7.6]{MVa1} to apply Rankin's trick to see that for some $\sigma>0$, 
\begin{align*}
%%\label{lambdansquare}
 \sum_{n \leq x, P(n) \leq x^{1/\log\log x}}\lambda(n)^{2} \leq  \sum_{n \leq x, P(n) \leq x^{1/\log\log x}}\lambda(n)^{2} (\frac {x}{n})^{\sigma} \leq  x^{\sigma} \sum_{P(n) \leq x^{1/\log\log x}}\frac {\lambda(n)^{2}}{n^{\sigma}}.
\end{align*}
  Note that by \eqref{lambdabound}, we see that when $\sigma>0$, 
\begin{align*}
%%\label{sumnsmooth}
  \sum_{P(n) \leq x^{1/\log\log x}}\frac {\lambda(n)^{2}}{n^{\sigma}}=\prod_{p \leq x^{1/\log\log x}}(1+\frac {\lambda(p)^{2}}{p^{\sigma}}+O(\frac 1{p^{2\sigma}}))
=\prod_{p \leq x^{1/\log\log x}}(1-\frac {\lambda(p)^{2}}{p^{\sigma}})^{-1}\cdot \prod_{p \leq x^{1/\log\log x}}(1+O(\frac 1{p^{2\sigma}})).
\end{align*}
  We now set $y=x^{1/\log\log x}$ and $\sigma=1-3/\log y$ to see that the last product above is convergent. Moreover, when $x$ is large enough, we have $2\sigma>1$, so that
\begin{align}
\label{sumsigmabound}
\begin{split}
  \sum_{P(n) \leq x^{1/\log\log x}}\frac {\lambda(n)^{2}}{n^{\sigma}}\ll & \prod_{p \leq x^{1/\log\log x}}(1-\frac {\lambda(p)^{2}}{p^{\sigma}})^{-1}=\exp \big(\sum_{p \leq y}-\log (1-\frac {\lambda(p)^{2}}{p^{\sigma}})\big)=\exp \big(\sum_{p \leq y}\frac {\lambda(p)^{2}}{p^{\sigma}}+O(\frac 1{p^{2\sigma}})\big) \\
\ll &
\exp \big(\sum_{p \leq y}\frac {\lambda(p)^{2}}{p^{\sigma}}\big).
\end{split}
\end{align}
  Note that when $1-4/\log y \leq \sigma \leq 1$ and $u \leq y$, we have $u^{1-\sigma}=\exp((1-\sigma)\log u)=1+O((1-\sigma)\log u)$. It follows from this, \eqref{merten1} and partial summation that
 \begin{align*}
  \sum_{p \leq y}\frac {\lambda(p)^{2}}{p^{\sigma}}=&\int^y_1u^{1-\sigma}d(\log\log u+O(1))=\int^y_1u^{1-\sigma}d\log\log u+O(1)
=\int^y_1(1+O((1-\sigma)\log u))d\log\log u+O(1) \\
=& \log \log y+O(1). 
\end{align*}

  We apply the above and \eqref{sumsigmabound} to see that
\begin{align*}
%%\label{sumsigmabound1}
\begin{split}
  \sum_{P(n) \leq x^{1/\log\log x}}\frac {\lambda(n)^{2}}{n^{\sigma}}
\ll & \log y.
\end{split}
\end{align*}
   We deduce from the above and \eqref{sumlambdapsquare} that 
\begin{align*}
%%\label{lambdansquare1}
 \sum_{n \leq x, P(n) \leq x^{1/\log\log x}}\lambda(n)^{2} \ll  x^{\sigma}\log y \ll \frac {x}{\log \log x}. 
\end{align*}
  We substitute the above estimation into \eqref{EPnsmall} to see that 
$$\Echar |\sum_{n \leq x, P(n) \leq x^{1/\log\log x}} \chi(n)\lambda(n)|^{2k} \ll (\frac {x}{\log \log x})^k, $$
  which gives a negligible contribution in Theorem \ref{lowerboundsfixedmodmean}.

 It therefore suffices to bound $\Echar |\sum_{n \leq x, P(n) > x^{1/\log\log x}} \chi(n)\lambda(n)|^{2k}$. Using the partition of unity $g_j$ from Lemma \ref{apres}, we may rewrite $\Echar |\sum_{n \leq x, P(n) > x^{1/\log\log x}} \chi(n)\lambda(n)|^{2k}$ as
\begin{eqnarray}
&& \Echar \prod_{i=-M}^{M} (\sum_{j=-N}^{N+1} g_{j}(S_{i,f}(\chi))) \Biggl|\sum_{\substack{n \leq x, \\ P(n) > x^{1/\log\log x}}} \chi(n)\lambda(n) \Biggr|^{2k} \nonumber \\
& = & \sum_{-N \leq j(-M), ..., j(0), ..., j(M) \leq N+1} \Echar \prod_{i=-M}^{M} g_{j(i)}(S_{i,f}(\chi)) \Biggl|\sum_{\substack{n \leq x, \\ P(n) > x^{1/\log\log x}}} \chi(n)\lambda(n) \Biggr|^{2k} \nonumber \\
& = & \sum_{-N \leq j(-M) , ... , j(0), ..., j(M) \leq N+1} \sigma(\textbf{j}) \E^{\textbf{j}} \Biggl|\sum_{\substack{n \leq x, \\ P(n) > x^{1/\log\log x}}} \chi(n)\lambda(n) \Biggr|^{2k} , \nonumber
\end{eqnarray}
where we set $\sigma(\textbf{j}) := \Echar \prod_{i=-M}^{M} g_{j(i)}(S_{i,f}(\chi))$ for any $(2M+1)$-vector $\textbf{j}$ from the outer sum above, and we define $\E^{\textbf{j}} W := \sigma(\textbf{j})^{-1} \Echar W \prod_{i=-M}^{M} g_{j(i)}(S_{i,f}(\chi))$ for all functions $W(\chi)$. Note that for the constant function $1$, we have $\E^{\textbf{j}} 1 = 1$ for all choices of the vector $\textbf{j}$. We then apply H\"older's inequality to $\E^{\textbf{j}}$ to see that 
\begin{align}
\label{EEj}
\begin{split} \Echar \Biggl|\sum_{\substack{n \leq x, \\ P(n) > x^{1/\log\log x}}} \chi(n)\lambda(n) \Biggr|^{2k} \leq \sum_{-N \leq j(-M) , ..., j(0), ..., j(M) \leq N+1} \sigma(\textbf{j})\Biggl( \E^{\textbf{j}} \Biggl|\sum_{\substack{n \leq x, \\ P(n) > x^{1/\log\log x}}} \chi(n)\lambda(n) \Biggr|^{2} \Biggr)^{k}. \end{split}
\end{align}

\subsection{Transition to random variables}
\label{subsecmaingorandom}

  We apply \eqref{alpha}, \eqref{alphalambdarel} to the complex numbers $a_m(p)=\lambda(p)p^{-im/\log^{1.01}P}$, $a_m(p^2)=(\lambda(p^2)-\lambda^2(p)/2)p^{-2im/\log^{1.01}P}$ to see that we have
\begin{align}
\label{abound}
\begin{split} 
 |a_m(p)|, |a_m(p^2)| \leq 3. 
\end{split}
\end{align}
  Moreover, we take $c(n)=\lambda(n)$ to see that 
\begin{align}
\label{lambdancond}
 \sum_{n \leq x} \tilde{d}(n) |\lambda(n)|^2 =\sum_{\substack{d \leq x \\ p \mid d \Rightarrow p \leq P}}\sum_{\substack{n \leq x \\ d|n}}\lambda^2(n).
\end{align} 

   We now consider the sum $\sum_{\substack{x/2 < n \leq x \\ d|n}}\lambda^2(n)$ to see that we have
\begin{align}
\label{lambdancond0}
 \sum_{\substack{x/2 < n \leq x \\ d|n}}\lambda^2(n) \ll \sum_{\substack{d|n}}\lambda^2(n)\Phi(\frac {n}{x}), 
\end{align} 
  where $\Phi$ is a smooth, non-negative function compactly supported on $[1/4, 3/2]$ satisfying $\Phi(x) \leq 1$ for all $x$ and $\Phi(x) =1$
for $x\in [1/2,1]$. The Mellin transform ${\widehat \Phi}(s)$ of $\Phi$ is defined for any complex number $s$ by
\begin{equation*}
%%\label{Phicheck}
{\widehat \Phi}(s) = \int\limits_{0}^{\infty} \Phi(x)x^{s}\frac {\dif x}{x}.
\end{equation*}
    Note that integration by parts shows that $\widehat{\Phi}(s)$ is a function satisfying the bound
\begin{align}
\label{boundsforphi}
  \widehat{\Phi}(s) \ll \min (1, |s|^{-1}(1+|s|)^{-E}),
\end{align}
for all $\Re(s) > 0$, and integers $E>0$. 

   Now the Mellin inversion leads to 
\begin{align}
\label{Perronprime1}
\begin{split}
    \sum_{\substack{n \\ d|n}}\lambda^2(n)\Phi(\frac {n}{x}) =&\frac 1{ 2\pi i}\int\limits_{(2)}(\sum_{\substack{n \geq 1 \\ c|n}}\frac {\lambda(n)^2}{n^{s}})\widehat{\Phi}(s) \frac{x^s}{s} \dif s. 
\end{split}
\end{align}
 We note that it is shown in the proof of \cite[Lemma 2.10]{G&Zhao25-11} that
\begin{align}
\label{Fd}
\begin{split}
 \sum_{\substack{n \geq 1 \\ d|n}}\frac {|\lambda(n)|^2}{n^{s}}=&\frac {F(s)}{d^s}\prod_{\substack{ p|d \\ p^{\nu_p} \| d}}\Big ( \sum_{j=0}^{\infty} \frac{|\lambda(p^{\nu_p+j})|^2}{p^{js}} \Big ) \frac{\zeta_p(2s)}{\zeta_p(s)L_p(s, \operatorname{sym}^2 f)}.
\end{split}
\end{align}
 where $L_p$ denotes the local factor at the prime $p$ in the Euler product of $L$ for any $L$-function and where $\zeta(s)$ is the Riemann zeta function,  $L(s, \operatorname{sym}^2 f)$ defined in \eqref{Lsymexp}. Also, it is shown in \cite[(2.25)]{G&Zhao25-11} that 
\begin{align}
\label{Frel1}
\begin{split}
 F(s)\zeta(2s)=\zeta(s)L(s, \operatorname{sym}^2 f).
\end{split}
\end{align}

   We now shift the contour of the integral in \eqref{Perronprime1} to $\Re(s)=1/2+\varepsilon$ to pick up a simple pole at $s=1$ of $\zeta(s)$. By \eqref{Frel1}, we see that the corresponding residue equals
\begin{align}
\label{lambdancond1}
\begin{split}
    \frac x{d}\prod_{\substack{ p|d \\ p^{\nu_p} \| d}}\Big ( \sum_{j=0}^{\infty} \frac{|\lambda(p^{\nu_p+j})|^2}{p^{j}} \Big ) \frac {\zeta_p(2) L(1, \operatorname{sym}^2 f)}{\zeta_p(1)L_p(1, \operatorname{sym}^2 f) \zeta(2)} =: g(d)x\frac {L(1, \operatorname{sym}^2 f)}{\zeta(2)},
\end{split}
\end{align}
  where one checks that $g$ is a multiplicative function satisfying $g(1)=1, g(p) =\lambda(p)^2/p+O(1/p^2), g(p^i)=O(1/p^i)$ for $i \geq 2$.

  Moreover, it follows from \eqref{boundsforphi} and \eqref{Fd} that the integral on the new line is
\begin{align}
\label{dfactorboundfurthersimplified}
\begin{split}
  \ll &  \frac {x^{1/2+\varepsilon}}{d^{1/2+\varepsilon}}\prod_{\substack{ p|d \\ p^{\nu_p} \| d}}\Big ( \sum_{j=0}^{\infty} \frac{|\lambda(p^{\nu_p+j})|^2}{p^{j(1/2+\varepsilon)}} \Big ) \ll   \frac {x}{d}\prod_{\substack{ p|d \\ p^{\nu_p} \| d}}\Big ( \sum_{j=0}^{\infty} \frac{|\lambda(p^{\nu_p+j})|^2}{p^{j(1/2+\varepsilon)}} \Big ) := g_1(d)x,
\end{split}
\end{align}
   where the middle estimation above follows from the observation that $d \leq x$ and one checks that $g_1$ is a multiplicative function satisfying $g_1(1)=1, g_1(p) =\lambda(p)^2/p+O(1/p^2), g_1(p^i)=O(1/p^i)$ for $i \geq 2$.

   We then deduce from \eqref{lambdancond0}, \eqref{lambdancond1} and \eqref{dfactorboundfurthersimplified} that we have
\begin{align*}
%%\label{lambdancond2}
 \sum_{\substack{x/2 < n \leq x \\ d|n}}\lambda^2(n) \ll g(d)x+  g_1(d)x. 
\end{align*} 
  We sum dyadically to derive from the above that
 \begin{align*}
%%\label{lambdancond3}
 \sum_{\substack{n \leq x \\ d|n}}\lambda^2(n) \ll g(d)x+  g_1(d)x. 
\end{align*} 

   It follows from this, \eqref{lambdancond} and the estimation $1+x \leq e^x$ for all real number $x$ that
\begin{align}
\label{lambdancond4}
\begin{split}
 & \sum_{n \leq x} \tilde{d}(n) |\lambda(n)|^2 \\
=& x\sum_{\substack{d \leq x \\ p \mid d \Rightarrow p \leq P}}g(d)+x\sum_{\substack{d \leq x \\ p \mid d \Rightarrow p \leq P}}g_1(d) \\
\ll & 
x\prod_{p | P}(1+g(p)+O(\frac 1{p^2}))+x\prod_{p | P}(1+g_1(p)+O(\frac 1{p^2})) \\
 \ll & x\exp \big(\sum_{p | P|}(|g(p)|+O(\frac 1{p^2}))\big)+ x\exp \big(\sum_{p | P|}(|g(p)|+O(\frac 1{p^2}))\big) \\
\ll & x\exp \big(\sum_{p | P|}\frac {\lambda^2(p)}{p}\big) \\
\ll & x \log P,
\end{split}
\end{align} 
  where the last estimation above follows from \eqref{merten1}.

  In view of \eqref{abound} and \eqref{lambdancond4}, we are in the position to apply Proposition \ref{condexpprop} with $Y=2M+1$ by setting $c_0=3$ there to see that when $x P^{4000((2M+1)/\delta)^2 \log(N\log P)} < q$, we have
$$ \E^{\textbf{j}} \Biggl|\sum_{\substack{n \leq x, \\ P(n) > x^{1/\log\log x}}} \chi(n)\lambda(n) \Biggr|^{2} = \frac{1}{\sigma(\textbf{j})} \Biggl( \E \prod_{i=-M}^{M} g_{j(i)}(S_{i,f}(h)) \Biggl|\sum_{\substack{n \leq x, \\ P(n) > x^{1/\log\log x}}} h(n)\lambda(n) \Biggr|^{2} + O\left(\frac{x}{(N \log P)^{(2M+1)(1/\delta)^2}} \right) \Biggr), $$
 where $h$ is a Steinhaus random multiplicative function and $S_{i,f}(h)$ is defined similar to $S_{m,f}(\chi)$ given in \eqref{Skf}. It follows from this and \eqref{EEj} that
\begin{align*}
%%\label{EEj1}
\begin{split} 
 & \Echar \Biggl|\sum_{\substack{n \leq x, \\ P(n) > x^{1/\log\log x}}} \chi(n)\lambda(n) \Biggr|^{2k} \\
\leq & \sum_{-N \leq j(-M) , ..., j(0), ..., j(M) \leq N+1} \sigma(\textbf{j})\Biggl( \frac{1}{\sigma(\textbf{j})} \Biggl( \E \prod_{i=-M}^{M} g_{j(i)}(S_{i,f}(h)) \Biggl|\sum_{\substack{n \leq x, \\ P(n) > x^{1/\log\log x}}} h(n)\lambda(n) \Biggr|^{2} + O\left(\frac{x}{(N \log P)^{(2M+1)(1/\delta)^2}} \right) \Biggr) \Biggr)^{k}  \\
\ll & \sum_{-N \leq j(-M) , ..., j(0), ..., j(M) \leq N+1} \sigma(\textbf{j})\Biggl( \frac{1}{\sigma(\textbf{j})} \E \prod_{i=-M}^{M} g_{j(i)}(S_{i,f}(h)) \Biggl|\sum_{\substack{n \leq x, \\ P(n) > x^{1/\log\log x}}} h(n)\lambda(n) \Biggr|^{2} \Biggr)^k \\
& +\sum_{-N \leq j(-M) , ..., j(0), ..., j(M) \leq N+1} \sigma(\textbf{j})\Biggl( \frac{1}{\sigma(\textbf{j})}O\left(\frac{x}{(N \log P)^{(2M+1)(1/\delta)^2}} \right)  \Biggr)^{k}.
\end{split}
\end{align*}

 Note that $\sum_{\textbf{j}} \sigma(\textbf{j}) = \Echar \prod_{i=-M}^{M} (\sum_{j=-N}^{N+1} g_{j}(S_{i,f}(\chi))) = \Echar 1 = 1$. We thus apply H\"{o}lder's inequality to see that
\begin{align*}
%%\label{EEjremainder}
\begin{split} 
 & \sum_{-N \leq j(-M) , ..., j(0), ..., j(M) \leq N+1} \sigma(\textbf{j})\Biggl( \frac{1}{\sigma(\textbf{j})}O\left(\frac{x}{(N \log P)^{(2M+1)(1/\delta)^2}} \right)  \Biggr)^{k} \\
\ll & \Biggl( \sum_{-N \leq j(-M) , ..., j(0), ..., j(M) \leq N+1} \sigma(\textbf{j}) \cdot \frac{1}{\sigma(\textbf{j})} \frac{x}{(N \log P)^{(2M+1)(1/\delta)^2}} \Biggr)^{k} = \Biggl( x (\frac{(2N+2)}{(N \log P)^{(1/\delta)^2}})^{2M+1} \Biggr)^{k}. 
\end{split}
\end{align*}
This gives a negligible contribution to \eqref{stpmaindisplay}. 

We now define $\sigma^{\text{rand}}(\textbf{j}) := \E \prod_{i=-M}^{M} g_{j(i)}(S_{i,f}(h))$ for all $(2M+1)$-vectors $\textbf{j}$ to see by Proposition \ref{boxprobprop} that $\sigma(\textbf{j})^{1-k} \ll \sigma^{\text{rand}}(\textbf{j})^{1-k} + (\frac{1}{(N \log P)^{(2M+1)(1/\delta)^2}})^{1-k}$. It follows that
\begin{align*}
%%\label{EEjremainder}
\begin{split} 
& \sum_{-N \leq j(-M) , ..., j(0), ..., j(M) \leq N+1} \sigma(\textbf{j})\Biggl( \frac{1}{\sigma(\textbf{j})} \E \prod_{i=-M}^{M} g_{j(i)}(S_{i,f}(h)) \Biggl|\sum_{\substack{n \leq x, \\ P(n) > x^{1/\log\log x}}} h(n)\lambda(n) \Biggr|^{2} \Biggr)^k  \\
 \ll & \sum_{\textbf{j}} \sigma^{\text{rand}}(\textbf{j})\Biggl( \frac{1}{\sigma^{\text{rand}}(\textbf{j})} \E \prod_{i=-M}^{M} g_{j(i)}(S_{i,f}(h)) \Biggl|\sum_{\substack{n \leq x, \\ P(n) > x^{1/\log\log x}}} h(n)\lambda(n)  \Biggr|^{2} \Biggr)^k  \\
& + (\frac{1}{(N \log P)^{(2M+1)(1/\delta)^2}})^{1-k} \sum_{\textbf{j}} \Biggl( \E \prod_{i=-M}^{M} g_{j(i)}(S_{i,f}(h)) \Biggl|\sum_{\substack{n \leq x, \\ P(n) > x^{1/\log\log x}}} h(n)\lambda(n)  \Biggr|^{2} \Biggr)^k. 
\end{split}
\end{align*}
We now apply H\"older's inequality again to the sum over $\textbf{j}$ and recall that the $g_j$ form a partition of unity to see that
\begin{align*}
%%\label{EEjremainder}
\begin{split} 
& (\frac{1}{(N \log P)^{(2M+1)(1/\delta)^2}})^{1-k} \sum_{\textbf{j}} \Biggl( \E \prod_{i=-M}^{M} g_{j(i)}(S_{i,f}(h)) \Biggl|\sum_{\substack{n \leq x, \\ P(n) > x^{1/\log\log x}}} h(n)\lambda(n)  \Biggr|^{2} \Biggr)^k \\
\ll & (\frac{1}{(N \log P)^{(2M+1)(1/\delta)^2}})^{1-k} \cdot ((2N+2)^{2M+1})^{1-k} \cdot \Biggl( \sum_{\textbf{j}} \E \prod_{i=-M}^{M} g_{j(i)}(S_{i,f}(h)) \Biggl|\sum_{\substack{n \leq x, \\ P(n) > x^{1/\log\log x}}} h(n)\lambda(n)  \Biggr|^{2} \Biggr)^k  \\
 = & \Biggl( (\frac{(2N+2)}{(N \log P)^{(1/\delta)^2}})^{2M+1} \Biggr)^{1-k} \Biggl( \E \Biggl|\sum_{\substack{n \leq x, \\ P(n) > x^{1/\log\log x}}} h(n)\lambda(n) \Biggr|^{2} \Biggr)^k  \\
=& \Biggl( (\frac{(2N+2)}{(N \log P)^{(1/\delta)^2}})^{2M+1} \Biggr)^{1-k} \Biggl( \sum_{\substack{n \leq x, \\ P(n) > x^{1/\log\log x}}} \lambda^2(n) \Biggr)^k \\
\ll & \Biggl( (\frac{(2N+2)}{(N \log P)^{(1/\delta)^2}})^{2M+1} \Biggr)^{1-k} x^k, 
\end{split}
\end{align*}
  where the last estimation above follows from \eqref{lambdasquareasymp}. As the last expression above is $ \ll e^{-(1-k)\log\log P} x^k$, its contribution to  \eqref{stpmaindisplay} is also negligible.

We conclude this section by noting that in order to establish Theorem \ref{lowerboundsfixedmodmean}, it remains to show that
\begin{equation}\label{stpnextdisplay}
\sum_{-N \leq j(-M) , ..., j(0), ..., j(M) \leq N+1} \sigma^{\text{rand}}(\textbf{j}) \Biggl( \E^{\textbf{j}, \text{rand}} \Biggl|\sum_{\substack{n \leq x, \\ P(n) > x^{1/\log\log x}}} h(n)\lambda(n) \Biggr|^{2} \Biggr)^{k} \ll \Biggl(\frac{x}{1 + (1-k)\sqrt{\log\log P}} \Biggr)^k.
\end{equation}
 Here we define $\E^{\textbf{j}, \text{rand}} W := \sigma^{\text{rand}}(\textbf{j})^{-1} \E W \prod_{i=-M}^{M} g_{j(i)}(S_{i,f}(h))$ for all random variables $W$. 

\subsection{Connecting to Euler products}
\label{subsecpasseuler}

As the various expressions $\prod_{i=-M}^{M} g_{j(i)}(S_{i,f}(h))$ in the definition of $\E^{\textbf{j}, \text{rand}}$ involve only the $h(p)$ for $p \leq P$ and the $h(p)$ are independent random variables with mean zero, we expand the square and keep in mind that $P < x^{1/\log\log x}$ to see that
\begin{align}
\label{smoothrandsplit}
\begin{split} 
 & \sum_{\textbf{j}} \sigma^{\text{rand}}(\textbf{j}) \Biggl( \E^{\textbf{j}, \text{rand}} \Biggl|\sum_{\substack{n \leq x, \\ P(n) > x^{1/\log\log x}}} h(n)\lambda(n) \Biggr|^{2} \Biggr)^{k} \\
=&  \sum_{\textbf{j}} \sigma^{\text{rand}}(\textbf{j}) \Biggl( \E^{\textbf{j}, \text{rand}} \sum_{\substack{m \leq x, \\ P(m) > x^{1/\log\log x} , \\ p|m \Rightarrow p > P}} \lambda^2(m)\Biggl|\sum_{\substack{n \leq x/m, \\ n \; \text{is} \; P \; \text{smooth}}} h(n)\lambda(n) \Biggr|^{2} \Biggr)^{k} \\
\ll & \Biggl(\E^{\textbf{j}, \text{rand}} \sum_{\substack{x^{1/\log\log x} < m \leq x, \\ p|m \Rightarrow p > P}} \frac{\lambda^2(m)X}{m} \int_{m}^{m(1+1/X)} |\sum_{\substack{n \leq x/t, \\ n \; \text{is} \; P \; \text{smooth}}} h(n)\lambda(n)|^{2} dt \Biggr)^k  \\
& + \Biggl(\E^{\textbf{j}, \text{rand}} \sum_{\substack{x^{1/\log\log x} < m \leq x, \\ p|m \Rightarrow p > P}} \frac{\lambda^2(m)X}{m} \int_{m}^{m(1+1/X)} |\sum_{\substack{x/t < n \leq x/m, \\ n \; \text{is} \; P \; \text{smooth}}} h(n)\lambda(n)|^{2} dt \Biggr)^k.
\end{split}
\end{align}
  where the last estimation above follows by replacing the condition that $P(m) > x^{1/\log\log x}$ with the weaker condition that $m > x^{1/\log\log x}$ and by introducing an integral to smooth out on a scale of $1/X$. Here and in what follows, we set $X = e^{\sqrt{\log x}}$. 

 As $\sum_{\textbf{j}} \sigma^{\text{rand}}(\textbf{j}) = \E \prod_{i=-M}^{M} (\sum_{j=-N}^{N+1} g_{j}(S_{i,f}(h))) = \E 1 = 1$, we apply H\"{o}lder's inequality to the sum over $\textbf{j}$ to see that 
\begin{align*}
%%\label{smoothrandsplit}
\begin{split} 
& \Biggl(\E^{\textbf{j}, \text{rand}} \sum_{\substack{x^{1/\log\log x} < m \leq x, \\ p|m \Rightarrow p > P}} \frac{\lambda^2(m)X}{m} \int_{m}^{m(1+1/X)} |\sum_{\substack{x/t < n \leq x/m, \\ n \; \text{is} \; P \; \text{smooth}}} h(n)\lambda(n)|^{2} dt \Biggr)^k \\
 \ll & \sum_{\textbf{j}} \sigma^{\text{rand}}(\textbf{j}) \Biggl( \E^{\textbf{j}, \text{rand}} \sum_{\substack{x^{1/\log\log x} < m \leq x, \\ p|m \Rightarrow p > P}} \frac{\lambda^2(m)X}{m} \int_{m}^{m(1+1/X)} |\sum_{\substack{x/t < n \leq x/m, \\ n \; \text{is} \; P \; \text{smooth}}} h(n)\lambda(n)|^{2} dt \Biggr)^k \\
 \leq & \Biggl(\sum_{\textbf{j}} \sigma^{\text{rand}}(\textbf{j}) \E^{\textbf{j}, \text{rand}} \sum_{\substack{x^{1/\log\log x} < m \leq x, \\ p|m \Rightarrow p > P}} \frac{\lambda^2(m)X}{m} \int_{m}^{m(1+1/X)} |\sum_{\substack{x/t < n \leq x/m, \\ n \; \text{is} \; P \; \text{smooth}}} h(n)\lambda(n)|^{2} dt \Biggr)^k. \end{split}
\end{align*}
 
  Note that $\sum_{\textbf{j}} \sigma^{\text{rand}}(\textbf{j})\E^{\textbf{j}, \text{rand}}W=\E W$ for all random variables $W$. We use this and apply the orthogonality of random multiplicative functions to see that the last expression above is
\begin{align}
\label{Erandnlarge}
\begin{split}  
 & \Biggl(\E \sum_{\substack{x^{1/\log\log x} < m \leq x, \\ p|m \Rightarrow p > P}} \frac{\lambda^2(m)X}{m} \int_{m}^{m(1+1/X)} |\sum_{\substack{x/t < n \leq x/m, \\ n \; \text{is} \; P \; \text{smooth}}} h(n)\lambda(n)|^{2} dt \Biggr)^k \\
=& \Biggl(\sum_{\substack{x^{1/\log\log x} < m \leq x, \\ p|m \Rightarrow p > P}} \frac{\lambda^2(m)X}{m} \int_{m}^{m(1+1/X)} \big(\sum_{\substack{x/t < n \leq x/m, \\ n \; \text{is} \; P \; \text{smooth}}} \lambda(n)^{2} \big)dt \Biggr)^k  \\
\ll & \Biggl(\sum_{\substack{x^{1/\log\log x} < m \leq x, \\ p|m \Rightarrow p > P}} \frac{\lambda^2(m)X}{m} \int_{m}^{m(1+1/X)} \big(\sum_{\substack{x/t < n \leq x/m}} \lambda(n)^{2} \big)dt \Biggr)^k  \\
\ll & \Biggl(\sum_{\substack{x^{1/\log\log x} < m \leq x, \\ p|m \Rightarrow p > P}} \frac{\lambda^2(m)X}{m} \int_{m}^{m(1+1/X)} \big(\frac xm-\frac xt+(\frac x{m})^{3/5} \big)dt \Biggr)^k,  
\end{split}
\end{align}
  where the last estimation above follows from \eqref{lambdasquareasymp}. Note that when $m \leq t \leq m(1+1/X)$, we have $\frac xm-\frac xt \leq \frac xm-\frac x{m(1+1/X)} \ll x/(mX)$. It follows that the last expression in \eqref{Erandnlarge} is 
\begin{align}
\label{EEjremaindersimplified}
\begin{split}  
 \ll & \Biggl(\sum_{\substack{x^{1/\log\log x} < m \leq x, \\ p|m \Rightarrow p > P}} \frac{\lambda^2(m)X}{m} \int_{m}^{m(1+1/X)} \big(\frac x{mX}+(\frac x{m})^{3/5} \big)dt \Biggr)^k
\ll  \Biggl(\sum_{\substack{x^{1/\log\log x} < m \leq x, \\ p|m \Rightarrow p > P}} \lambda^2(m)(\frac{x}{mX} + (\frac x{m})^{3/5}) \Biggr)^k. 
\end{split}
\end{align}

  By \eqref{lambdasquareasymp} and partial summation, we see that
\begin{align*}
%%\label{EEjremainder}
\begin{split}  
 \sum_{\substack{m \leq x}} \frac{\lambda^2(m)}{m}\ll \log x.
\end{split}
\end{align*}
  It follows from this and the fact that $X=e^{\sqrt{\log x}}$ that
\begin{align}
\label{EEjremainder1}
\begin{split}  
 \sum_{\substack{x^{1/\log\log x} < m \leq x, \\ p|m \Rightarrow p > P}} \frac{x\lambda^2(m)}{mX} \ll \frac{x}{X} \sum_{\substack{m \leq x}} \frac{\lambda^2(m)}{m} \ll \frac {x\log x}{X} \ll \frac {x}{\log P}.
\end{split}
\end{align}
  
  For $x^{1/10} \leq u \leq x$, we now use the sieve method to study the sum
\begin{equation*}
      \sum_{\substack{x^{1/\log\log x} < m \leq u, \\ p|m \Rightarrow p > P}}  \lambda(m)^2.
\end{equation*}
 We apply the upper bound part of \cite[Theorem 12.5]{FI10} for the sequence $\mathcal{A}=\{m: x^{1/\log\log x} < m \leq u  \}$ with $z:=P<u^{1/10}, D=u^{1/20}$ and $Y=L(1, \operatorname{sym}^2 f)\zeta^{-1}(2)(u-x^{1/\log\log x}) \ll u$.  By \cite[Lemma 2.14]{G&Zhao25-11}, we see that
\begin{align}
\label{Ad}
\begin{split}
    &  \sum_{\substack{m \in \mathcal{A} \\ d|m}} \lambda(m)^2=g(d)Y+O(u^{3/4+\varepsilon}),
\end{split}
\end{align}
  where $g$ is defined in \eqref{lambdancond}. It then follows from \eqref{merten1} that the sieve dimension in our case (see \cite[(5.35)]{FI10}) is $\kappa=1$.
We now apply \cite[ Theorem 11.13]{FI10} to see that
\begin{equation*}
    \sum_{\substack{x^{1/\log\log x} < m \leq u, \\ p|m \Rightarrow p > P}}  \lambda(m)^2 \ll YV(z)(F(s)+O((\log D)^{-1/6}))+R(D,z).
\end{equation*}
  Here by \cite[(5.36)]{FI10} and the observation that $\kappa=1$ in our case, 
\begin{align*}
%%\label{Vz}
\begin{split}
    &  V(z)=\prod_{p \leq z}(1-g(p)) \ll \frac 1{\log P}.
\end{split}
\end{align*}

  Also, by \eqref{Ad} and the expression for $R(D,z)$ given on \cite[p. 207]{FI10},
\begin{align*}
%%\label{RDz}
\begin{split}
    &  R(D,z) \ll u^{3/4+\varepsilon}\sum_{d | \mathcal{P}(z)}1.
\end{split}
\end{align*}
Since $2 \leq z \leq \sqrt{D}$ and $\kappa=1$, we can apply \cite[Lemma 12.3]{FI10}, getting
\begin{align*}
%%\label{RDzest}
\begin{split}
    &  R(D,z) \ll u^{3/4+\varepsilon}\sum_{d | \mathcal{P}(z)}1 \ll u^{3/4+\varepsilon}D \ll u^{4/5+\varepsilon}.
\end{split}
\end{align*}

   Moreover, we have $s=\log D/\log z$ is large when $x$ is large enough, so by \cite[(11.134)]{FI10}, we see that $F(s)\ll 1$ when $x$ is large enough, where $
F$ is the function defined by delayed differential equations in \cite[(12.1)--(12.2)]{FI10}. We then conclude that for $x^{1/10} \leq u \leq x$,
\begin{equation}
\label{lambdasquareplarge}
     \sum_{\substack{x^{1/\log\log x} < m \leq u, \\ p|m \Rightarrow p > P}} \lambda(m)^2\ll \frac{Y}{\log P} \ll \frac{u}{\log P}.
\end{equation}
  We apply the above and partial summation to see that
\begin{align}
\label{EEjremainder21}
\begin{split}  
  \sum_{\substack{x^{1/10} < m \leq x, \\ p|m \Rightarrow p > P}} \lambda^2(m)(\frac x{m})^{3/5} \ll  \frac {x}{\log P}.
\end{split}
\end{align}

   On the other hand, we apply \eqref{lambdabound} and sum trivially to see that
\begin{align}
\label{EEjremainder22}
\begin{split}  
   \sum_{\substack{x^{1/\log\log x} < m \leq x^{1/10}, \\ p|m \Rightarrow p > P}} \lambda^2(m)(\frac x{m})^{3/5} \ll  x^{3/5}x^{1/10} \ll \frac {x}{\log P}.
\end{split}
\end{align}

   We derive from \eqref{EEjremainder21} and \eqref{EEjremainder22} that
\begin{align*}
%%\label{EEjremainder2}
\begin{split}  
   \sum_{\substack{x^{1/\log\log x} < m \leq x, \\ p|m \Rightarrow p > P}} \lambda^2(m)(\frac x{m})^{3/5}\ll \frac {x}{\log P}.
\end{split}
\end{align*}  
  The above together with \eqref{EEjremainder1} now implies that the last expression in \eqref{EEjremaindersimplified} is $\ll \frac {x}{\log P}$. We further deduce from this and \eqref{smoothrandsplit}--\eqref{EEjremaindersimplified} that the total contribution to  \eqref{stpnextdisplay} from the second bracket in the last expression of \eqref{smoothrandsplit} is
\begin{align}
\label{2ndcontri}
\begin{split} 
  \ll & \sum_{\textbf{j}} \sigma^{\text{rand}}(\textbf{j}) \Biggl(\frac {x}{\log P}\Biggr)^k=\Biggl(\frac {x}{\log P}\Biggr)^k.
\end{split}
\end{align}
  This contribution to  \eqref{stpnextdisplay} is therefore negligible.
 
  On the other hand, we that the total contribution to \eqref{stpnextdisplay} from the first bracket in the last expression of \eqref{smoothrandsplit} is
\begin{align}
\label{1stcontribution}
\begin{split} 
 \ll & \sum_{\textbf{j}} \sigma^{\text{rand}}(\textbf{j}) \Biggl(\E^{\textbf{j}, \text{rand}} \int_{x^{1/\log\log x}}^{x(1+1/X)} |\sum_{\substack{n \leq x/t, \\ n \; \text{is} \; P \; \text{smooth}}} h(n)\lambda(n)|^2 \sum_{\substack{t/(1+1/X) < m \leq t, \\ p|m \Rightarrow p > P}} \frac{\lambda^2(m)X}{m} dt \Biggr)^k. 
\end{split}
\end{align}

  Note that, similar to \eqref{lambdasquareplarge}, we have
\begin{align*}
  \sum_{\substack{t/(1+1/X) < m \leq t, \\ p|m \Rightarrow p > P}} \frac{\lambda^2(m)X}{m} \ll \frac{X}{t}\sum_{\substack{t/(1+1/X) < m \leq t, \\ p|m \Rightarrow p > P}} \lambda^2(m) \ll \frac{X}{t} \cdot (t-\frac {t}{1+1/X})\frac 1{\log P} \ll \frac 1{\log P}.
\end{align*}
  It follows from this that the expression in \eqref{1stcontribution} is
\begin{align*}
%%\label{1stcontribution1}
\begin{split} 
\ll & \sum_{\textbf{j}} \sigma^{\text{rand}}(\textbf{j}) (\frac{1}{\log P})^q \Biggl(\E^{\textbf{j}, \text{rand}} \int_{x^{1/\log\log x}}^{x(1+1/X)} |\sum_{\substack{n \leq x/t, \\ n \; \text{is} \; P \; \text{smooth}}} h(n)\lambda(n)|^2 dt \Biggr)^k  \\
 = & \sum_{\textbf{j}} \sigma^{\text{rand}}(\textbf{j}) (\frac{x}{\log P})^q \Biggl(\E^{\textbf{j}, \text{rand}} \int_{1+1/X}^{x^{1-1/\log\log x}} |\sum_{\substack{n \leq z, \\ n \; \text{is} \; P \; \text{smooth}}} h(n)\lambda(n)|^2 \frac{dz}{z^2} \Biggr)^k,
\end{split}
\end{align*}  
 where the equality above follows by a substitution $z=x/t$ in the first integral above. We now apply Lemma \ref{parseval} to see the last expression above is
\begin{equation}\label{produpperintrand}
\ll (\frac{x}{\log P})^k \sum_{\textbf{j}} \sigma^{\text{rand}}(\textbf{j}) (\E^{\textbf{j}, \text{rand}} \int_{-\infty}^{\infty} \frac{|F_{P,f}^{\text{rand}}(1/2 + it)|^2}{|1/2+it|^2} dt)^k ,
\end{equation}
where $F_{P,f}^{\text{rand}}(s)$ is defined as in \eqref{FPdef}. 

As $0<k \leq 1$, we apply the estimation that $(\sum^{\infty}_{i=1}a_i)^k \leq \sum^{\infty}_{i=1}a^k_i$ for convergent series $\sum^{\infty}_{i=1}a_i$ with $a_i \geq 0$ to see that we may divide the integral in \eqref{produpperintrand} into sub-intervals of length $1$ so that the expression in \eqref{produpperintrand} is
\begin{eqnarray}
& \leq & (\frac{x}{\log P})^k \sum_{\textbf{j}} \sigma^{\text{rand}}(\textbf{j}) \sum_{v=-\infty}^{\infty} (\E^{\textbf{j}, \text{rand}} \int_{v-1/2}^{v+1/2} \frac{|F_{P,f}^{\text{rand}}(1/2 + it)|^2}{|1/2+it|^2} dt)^k \nonumber \\
& \ll & (\frac{x}{\log P})^k \sum_{v=-\infty}^{\infty} \frac{1}{(|v|+1)^{4/3}} \sum_{\textbf{j}} \sigma^{\text{rand}}(\textbf{j}) (\E^{\textbf{j}, \text{rand}} \int_{v-1/2}^{v+1/2} |F_{P,f}^{\text{rand}}(1/2 + it)|^2 dt)^k . \nonumber
\end{eqnarray}

  We apply H\"{o}lder's inequality and the orthogonality of random multiplicative functions to see that the sum over the terms with $|v| > \log^{0.01}P$ in the last expression above is
\begin{align*}
%%\label{1stcontribution1}
\begin{split}
 \leq & (\frac{x}{\log P})^k \sum_{|v| > \log^{0.01}P} \frac{1}{(|v|+1)^{4/3}} (\sum_{\textbf{j}} \sigma^{\text{rand}}(\textbf{j}) \E^{\textbf{j}, \text{rand}} \int_{v-1/2}^{v+1/2} |F_{P,f}^{\text{rand}}(1/2 + it)|^2 dt)^k  \\
 = & (\frac{x}{\log P})^k \sum_{|v| > \log^{0.01}P} \frac{1}{(|v|+1)^{4/3}} (\E \int_{v-1/2}^{v+1/2} |F_{P,f}^{\text{rand}}(1/2 + it)|^2 dt)^k \\
 =& (\frac{x}{\log P})^q \sum_{|v| > \log^{0.01}P} \frac{1}{(|v|+1)^{4/3}} ( \sum_{\substack{n = 1, \\ n \; \text{is} \; P \; \text{smooth}}}^{\infty} \frac{\lambda^2(n)}{n} )^k \\
\ll &  (\frac{x}{\log P})^k \frac{1}{\log^{1/300}P} \log^{k}P, 
\end{split}
\end{align*} 
  where the last estimation above follows from \eqref{lambdansquare0}. This contribution to  \eqref{stpnextdisplay} is also negligible.

 It follows from \eqref{2ndcontri}, \eqref{produpperintrand} and our discussions above that in order to establish \eqref{stpnextdisplay}, it suffices to show that we have uniformly for all $|v| \leq \log^{0.01}P$, 
\begin{equation}\label{tobeprovedeqrand}
\sum_{\textbf{j}} \sigma^{\text{rand}}(\textbf{j}) (\E^{\textbf{j}, \text{rand}} \int_{v-1/2}^{v+1/2} |F_{P,f}^{\text{rand}}(1/2 + it)|^2 dt)^k \ll \Biggl(\frac{\log P}{1 + (1-k)\sqrt{\log\log P}}\Biggr)^{k} .
\end{equation}

   As the treatments are similar, we consider only the case $v=0$ in \eqref{tobeprovedeqrand} in what follows. 
In this case, we have
\begin{eqnarray}
&& \sum_{\textbf{j}} \sigma^{\text{rand}}(\textbf{j}) (\E^{\textbf{j}, \text{rand}} \int_{-1/2}^{1/2} |F_{P,f}^{\text{rand}}(1/2 + it)|^2 dt)^k \nonumber \\
& \ll & \sum_{\textbf{j}} \sigma^{\text{rand}}(\textbf{j}) ( \E^{\textbf{j}, \text{rand}} \sum_{|m| \leq \frac{\log^{1.01}P}{2}} \int_{-\frac{1}{2\log^{1.01}P}}^{\frac{1}{2\log^{1.01}P}} |F_{P,f}^{\text{rand}}(\frac{1}{2} + i\frac{m}{\log^{1.01}P}+it) - F_{P,f}^{\text{rand}}(\frac{1}{2} + i\frac{m}{\log^{1.01}P})|^2 dt )^k \nonumber \\
&& + \sum_{\textbf{j}} \sigma^{\text{rand}}(\textbf{j}) (\E^{\textbf{j}, \text{rand}} \frac{1}{\log^{1.01}P} \sum_{|m| \leq (\log^{1.01}P)/2} |F_{P,f}^{\text{rand}}(1/2 + i\frac{m}{\log^{1.01}P})|^2 )^k. \nonumber
\end{eqnarray}
  We now apply H\"older's inequality to the sum over $\textbf{j}$ and Lemma \ref{disccontlem} to see that the first expression on the right-hand side above is
$$ \leq \Biggl( \E \sum_{|m| \leq \frac{\log^{1.01}P}{2}} \int_{-\frac{1}{2\log^{1.01}P}}^{\frac{1}{2\log^{1.01}P}} |F_{P,f}^{\text{rand}}(1/2 + i\frac{k}{\log^{1.01}P} + it) - F_{P,f}^{\text{rand}}(1/2 + i\frac{m}{\log^{1.01}P})|^2 dt \Biggr)^k \ll \log^{0.99q}P. $$
This gives a negligible contribution to \eqref{tobeprovedeqrand}. We thus conclude that in order to complete the proof of Theorem \ref{lowerboundsfixedmodmean},  it suffices to show that
\begin{equation}\label{tobeprovedeqrand1}
\sum_{\textbf{j}} \sigma^{\text{rand}}(\textbf{j}) (\E^{\textbf{j}, \text{rand}} \frac{1}{\log^{1.01}P} \sum_{|m| \leq (\log^{1.01}P)/2} |F_{P,f}^{\text{rand}}(1/2 + i\frac{m}{\log^{1.01}P})|^2 )^k \ll \Biggl(\frac{\log P}{1 + (1-k)\sqrt{\log\log P}}\Biggr)^{k} .
\end{equation}

\subsection{Completion of the proof}\label{subsecconclusion}

We now define
$$ \mathcal{T} := \{m \in \mz : |F_{P,f}^{\text{rand}}(1/2 + i\frac{m}{\log^{1.01}P})| \geq \log^{1.1}P \;\;\; \text{or} \;\;\; |F_{P,f}^{\text{rand}}(1/2 + i\frac{m}{\log^{1.01}P})| \leq \frac{1}{\log^{1.1}P} \}. $$
 We divide the sum over $|m| \leq (\log^{1.01}P)/2$ according to whether $m \in \mathcal{T}$ or not, and then apply H\"{o}lder's inequality to see that
\begin{align}
\label{Etau}
\begin{split}
& \sum_{\textbf{j}} \sigma^{\text{rand}}(\textbf{j}) (\E^{\textbf{j}, \text{rand}} \frac{1}{\log^{1.01}P} \sum_{|m| \leq (\log^{1.01}P)/2} |F_{P,f}^{\text{rand}}(1/2 + i\frac{m}{\log^{1.01}P})|^2 )^k  \\
 \leq & \sum_{\textbf{j}} \sigma^{\text{rand}}(\textbf{j}) (\E^{\textbf{j}, \text{rand}} \frac{1}{\log^{1.01}P} \sum_{\substack{|m| \leq (\log^{1.01}P)/2, \\ m \notin \mathcal{T}}} |F_{P,f}^{\text{rand}}(1/2 + i\frac{m}{\log^{1.01}P})|^2 )^k \\
& + \Biggl( \frac{1}{\log^{1.01}P} \sum_{|m| \leq (\log^{1.01}P)/2} \E \textbf{1}_{m \in \mathcal{T}} |F_{P,f}^{\text{rand}}(1/2 + i\frac{m}{\log^{1.01}P})|^2 \Biggr)^k. 
\end{split}
\end{align} 

  Note that
\begin{align}
\label{E1tau}
\begin{split}
    \E \textbf{1}_{m \in \mathcal{T}} |F_{P,f}^{\text{rand}}(1/2 + i\frac{m}{\log^{1.01}P})|^2 \ll (\log^{1.1}P)^{-0.2} \E |F_{P,f}^{\text{rand}}(1/2 + i\frac{m}{\log^{1.01}P})|^{2.2} + (\frac{1}{\log^{1.1}P})^2
\end{split}
\end{align}

We now treat $\E |F_{P,f}^{\text{rand}}(1/2 + i\frac{m}{\log^{1.01}P})|^{2.2}$ by noting first that we have the trivial bound
\begin{equation*}
    \prod_{p\leq 400}\Big| 1-\frac{\alpha_ph(p)}{p^{1/2+i\frac{m}{\log^{1.01}P}}}\Big|^{-2.2}\Big| 1-\frac{\beta_ph(p)}{p^{1/2+i\frac{m}{\log^{1.01}P}}}\Big|^{-2.2} \ll 1.
\end{equation*}
 We then apply \eqref{Eest0} with $\alpha=1.1,\beta=0, \sigma_1=\sigma_2=0, t_1=t_2=\frac{m}{\log^{1.01}P}$ and $z=400$ together with \eqref{alphalambdarel}, \eqref{merten1} to see that
\begin{equation*}
\begin{split}
   & \E |F_{P,f}^{\text{rand}}(1/2 + i\frac{m}{\log^{1.01}P})|^{2.2} \ll   \exp\bigg(\sum_{400<p\leq P}\frac{1.1^2\lambda^2(p)}{p}\bigg) \ll \log^{1.21}P. 
\end{split}
\end{equation*}
  It follows from this and \eqref{E1tau} that
\begin{align*}
%%\label{E1tau}
\begin{split}
    \E \textbf{1}_{m \in \mathcal{T}} |F_{P,f}^{\text{rand}}(1/2 + i\frac{m}{\log^{1.01}P})|^2 \ll (\log^{1.1}P)^{-0.2} \log^{1.21}P + (\frac{1}{\log^{1.1}P})^2 \ll \log^{0.99}P. 
\end{split}
\end{align*}
  We deduce from this that
\begin{align*}
%%\label{E1tau}
\begin{split}
    \Biggl( \frac{1}{\log^{1.01}P} \sum_{|m| \leq (\log^{1.01}P)/2} \E \textbf{1}_{m \in \mathcal{T}} |F_{P,f}^{\text{rand}}(1/2 + i\frac{m}{\log^{1.01}P})|^2 \Biggr)^k \ll \Biggl( \frac{1}{\log^{0.02}P}\Biggr)^k . 
\end{split}
\end{align*} 
 The above gives an negligible contribution to \eqref{tobeprovedeqrand1}.

Next, we note that the contribution to \eqref{tobeprovedeqrand1} from the sum over $m \notin \mathcal{T}$ in \eqref{Etau} is
\begin{eqnarray}
& \leq & \sum_{\textbf{j}} \sigma^{\text{rand}}(\textbf{j}) (\E^{\textbf{j}, \text{rand}} \frac{1}{\log^{1.01}P} \sum_{\substack{|m| \leq (\log^{1.01}P)/2, \\ k \notin \mathcal{T}}} \textbf{1}_{|S_{m,f}(h) - j(m)| \leq 1} |F_{P,f}^{\text{rand}}(1/2 + i\frac{m}{\log^{1.01}P})|^2 )^k + \nonumber \\
&& + \sum_{\textbf{j}} \sigma^{\text{rand}}(\textbf{j}) (\E^{\textbf{j}, \text{rand}} \frac{1}{\log^{1.01}P} \sum_{|m| \leq (\log^{1.01}P)/2} \textbf{1}_{|S_{k,f}(h) - j(m)| > 1} \textbf{1}_{m \notin \mathcal{T}} |F_{P,f}^{\text{rand}}(1/2 + i\frac{m}{\log^{1.01}P})|^2 )^k. \nonumber
\end{eqnarray}

  We now applying H\"{o}lder's inequality to the sum over $\textbf{j}$, and recalling the definitions of $\E^{\textbf{j}, \text{rand}}$ and $\sigma^{\text{rand}}(\textbf{j})$ to see that
\begin{align*}
%%\label{E1tau}
\begin{split}
 & \sum_{\textbf{j}} \sigma^{\text{rand}}(\textbf{j}) (\E^{\textbf{j}, \text{rand}} \frac{1}{\log^{1.01}P} \sum_{|m| \leq (\log^{1.01}P)/2} \textbf{1}_{|S_{k,f}(h) - j(m)| > 1} \textbf{1}_{m \notin \mathcal{T}} |F_{P,f}^{\text{rand}}(1/2 + i\frac{m}{\log^{1.01}P})|^2 )^k\\
\ll & (\frac{1}{\log^{1.01}P} \sum_{|m| \leq \frac{\log^{1.01}P}{2}} \sum_{\textbf{j}} \E \prod_{i=-M}^{M} g_{j(i)}(S_{i,f}(h)) \cdot \textbf{1}_{|S_{m,f}(h) - j(m)| > 1} \textbf{1}_{m \notin \mathcal{T}} |F_{P,f}^{\text{rand}}(1/2 + i\frac{m}{\log^{1.01}P})|^2 )^k. 
\end{split}
\end{align*}

  It follows from the discussions in Section 3.5 of \cite{Harper23} that when $N \geq 1.2\log\log P$, we have by Lemma \ref{apres} that 
\begin{align*}
%%\label{E1tau}
\begin{split}
 & (\frac{1}{\log^{1.01}P} \sum_{|m| \leq \frac{\log^{1.01}P}{2}} \sum_{\textbf{j}} \E \prod_{i=-M}^{M} g_{j(i)}(S_{i,f}(h)) \cdot \textbf{1}_{|S_{m,f}(h) - j(m)| > 1} \textbf{1}_{m \notin \mathcal{T}} |F_{P,f}^{\text{rand}}(1/2 + i\frac{m}{\log^{1.01}P})|^2 )^k \\
\ll & (\frac{\delta}{\log^{1.01}P} \sum_{|m| \leq \frac{\log^{1.01}P}{2}} \sum_{-N \leq j(m) \leq N+1} \E |F_{P,f}^{\text{rand}}(1/2 + i\frac{m}{\log^{1.01}P})|^2 )^k \\
 \ll & (\delta N \sum_{\substack{n = 1, \\ n \; \text{is} \; P \; \text{smooth}}}^{\infty} \frac{\lambda^2(n)}{n} )^k   \ll  (\delta N \log P)^k,
\end{split}
\end{align*}
 where the last estimation above follows from \eqref{lambdansquare0}. This will leads to a negligible contribution to \eqref{tobeprovedeqrand1} provided we choose $\delta$ so that $\delta \leq \frac{1}{N\sqrt{\log\log P}}$.

  In view of the above discussions, we see that in order to establish \eqref{tobeprovedeqrand1}, it suffices to show that
\begin{eqnarray}
&& \sum_{\textbf{j}} \sigma^{\text{rand}}(\textbf{j}) \Biggl(\E^{\textbf{j}, \text{rand}} \frac{1}{\log^{1.01}P} \sum_{\substack{|m| \leq (\log^{1.01}P)/2, \\ k \notin \mathcal{T}}} \textbf{1}_{|S_{m,f}(h) - j(m)| \leq 1} |F_{P,f}^{\text{rand}}(1/2 + i\frac{m}{\log^{1.01}P})|^2 \Biggr)^k \nonumber \\
& \ll & \Biggl(\frac{\log P}{1 + (1-q)\sqrt{\log\log P}}\Biggr)^{k} . \nonumber
\end{eqnarray}

  We now recall from \eqref{FPdef} and \eqref{Skf} that $|F_{P,f}^{\text{rand}}(1/2 + i\frac{m}{\log^{1.01}P})| = \exp\Big (-\Re \sum_{p \leq P} \log(1 - \frac{\alpha_ph(p)}{p^{1/2 + i\frac{m}{\log^{1.01}P}}})-\Re \sum_{p \leq P} \log(1 - \frac{\beta_ph(p)}{p^{1/2 + i\frac{m}{\log^{1.01}P}}})\Big ) \asymp \exp\Big (
S_{m,f}(h)\Big )$ by \eqref{alphalambdarel} for all $m$ and $h(n)$. We then proceed as in Section 3.6 of \cite{Harper23} upon using Lemma \ref{apres} to see that
\begin{align*}
%%\label{E1tau}
\begin{split}
 & \sum_{\textbf{j}} \sigma^{\text{rand}}(\textbf{j}) \Biggl(\E^{\textbf{j}, \text{rand}} \frac{1}{\log^{1.01}P} \sum_{\substack{|m| \leq (\log^{1.01}P)/2, \\ k \notin \mathcal{T}}} \textbf{1}_{|S_{m,f}(h) - j(m)| \leq 1} |F_{P,f}^{\text{rand}}(1/2 + i\frac{m}{\log^{1.01}P})|^2 \Biggr)^k \\
 \ll & \E( \frac{1}{\log^{1.01}P} \sum_{|m| \leq (\log^{1.01}P)/2} |F_{P,f}^{\text{rand}}(1/2 + i\frac{m}{\log^{1.01}P})|^2 )^k+(\delta N \log^{2.2}P)^k \\
\ll & \left(\frac{\log P}{1 + (1-k)\sqrt{\log\log P}}\right)^{k}+(\delta N \log^{2.2}P)^k,
\end{split}
\end{align*}
 where the last estimation above follows from Lemma \ref{mcres2}. The above estimation now allows us to obtain the bound given in \eqref{tobeprovedeqrand1}  as long as we have $\delta \leq \frac{1}{N \log^{1.2}P \sqrt{\log\log P}}$. 

Recall now we needed to have $x P^{4000((2M+1)/\delta)^2 \log(N\log P)} < q$ in order to apply Proposition \ref{condexpprop} in section \ref{subsecmaingorandom}. As we have $M = 2\log^{1.02}P$, the above estimation will hold provided that $P^{10^5 (\log^{2.04}P) (1/\delta)^2 \log(N\log P)} < q/x$. Recall also from our discussions above that we need to have $N \geq 1.2\log\log P$ and $\delta \leq \frac{1}{N\sqrt{\log\log P}}$, as well as $\delta \leq \frac{1}{N \log^{1.2}P \sqrt{\log\log P}}$.

We now set $N = \lceil 1.2\log\log P \rceil$, and $\delta = \frac{1}{\log^{1.3}P}$ to see that for these values, we have $P^{10^5 (\log^{2.04}P) (1/\delta)^2 \log(N\log P)} < P^{10^6(\log^{4.64}P) \log\log P}$. Recall that $P$ is slightly smaller than $\exp(\log^{1/6}L)$ so that we have $P^{10^6(\log^{4.64}P) \log\log P} \leq \exp(10^7 (\log^{5.64/6} L)\log \log L) \leq \exp(\log L) \leq q/x$. As all the other conditions above are satisfied, this implies the estimation given in  \eqref{tobeprovedeqrand1} and hence completes the proof of Theorem \ref{lowerboundsfixedmodmean}.

\vspace*{.5cm}

\noindent{\bf Acknowledgments.} P. Gao is supported in part by NSFC grant 12471003, X. Wu is supported in part by NSFC grant 12271135 and the Fundamental Research Funds for the Central Universities grant JZ2025HGTG0254.

\bibliography{biblio}
\bibliographystyle{amsxport}

\end{document}